\documentclass[12pt,titlepage]{article} 
\usepackage{amsmath}                                                                         
\usepackage{amsfonts,amssymb,amsthm}                                                         
\usepackage[all,knot,poly]{xy}
\usepackage{graphicx}

\newtheorem{theorem}{Theorem}[section]
\newtheorem{theorem*}{Theorem}
          
\newtheorem{proposition}[theorem]{Proposition}                                        
\newtheorem{corollary}[theorem]{Corollary} 

\theoremstyle{definition}
\newtheorem{definition}[theorem]{Definition}
\newtheorem{example}[theorem]{Example}

\theoremstyle{remark}

\numberwithin{equation}{section}     
\allowdisplaybreaks[1]                                                     

\newcommand{\iso}{\cong} %isomorphism symbol                                                 
\newcommand{\bC}{{\mathbf C}}

\newcommand{\bZ}{{\mathbf Z}}
\DeclareMathOperator{\Tr}{Tr}
\providecommand{\abs}[1]{\lvert#1\rvert}
\providecommand{\Babs}[1]{\Bigl\lvert#1\Bigr\rvert}
\newcommand{\ra}{\rightarrow}                                                                
                                                           
\begin{document}                                                                             
\title{Polynomial Invariants, Knot Homologies, and Higher Twist Numbers of Weaving Knots $W(3,n)$}
% Primary 57M25   Secondary 57M27
% Keywords  Hecke algebra, trace function, knot signature, Khovanov homology, polynomial invariants, braid group representations, Heegard-Floer
% homology
\author{
Rama Mishra
\thanks{We thank the office of the Dean of the College of Arts and Sciences,
  New Mexico State University, for arranging a visiting appointment.}
\\
Department of Mathematics
\\
IISER Pune
\\
Pune, India
\and 
Ross Staffeldt 
\thanks{We thank the office of the Vice-President for Research, New Mexico
  State University, for arranging a visit to IISER-Pune to discuss
  implementation of an MOU between IISER-Pune and NMSU.}
\\Department of Mathematical Sciences\\
New Mexico State University\\
Las Cruces, NM 88003 USA}
\maketitle
\begin{abstract}
We investigate several conjectures in  
geometric topology by assembling computer data obtained by studying  
weaving knots, a doubly infinite family $W(p,n)$ of examples of  
hyperbolic knots. 
 In  particular, we compute some important polynomial knot invariants, as well as knot homologies, for the subclass $W(3,n)$ of this  
 family. 
We use these knot invariants to conclude that all knots $W(3,n)$ are  fibered
knots and provide estimates for some geometric invariants of these
knots.  Finally,  we study the asymptotics of the ranks of their Khovanov homology groups. 
Our investigations provide evidence for our conjecture that, asymptotically as $n$ grows large,
 the ranks of Khovanov homology groups of $W(3,n)$ are normally distributed. 
\end{abstract}
\section{Introduction} \label{Introduction}

Distinguishing knots and links up to ambient isotopy is the central problem in knot theory. 
The recipe that a knot theorist uses is to compute some knot invariants and see if one of them can be of help.
Classically one used  the Alexander polynomial, which has a topological
definition
\cite[p.160]{Rolfsen}, to distinguish knots.
Over the last 35 years tremendous progress has been made in the development of several new knot 
 invariants, starting with the Jones polynomial and the HOMFLY-PT polynomial
\cite{Jones_poly86}.
Recently even more sophisticated invariants such as Heegard-Floer homology groups \cite{Manolescu} 
and Khovanov homology groups \cite{BarNatan_Categorification} have been added to the toolkit.
These homology theories are called {\it categorifications} of polynomial invariants, because 
passage to an appropriate Euler characteristic retrieves a polynomial invariant.
Though these invariants are well understood to a great extent, 
 it is very difficult to compute them for a given knot whose crossing number is only moderately large.
Moreover, some of these invariants yield a very large amount of numerical data, and the problem  arises of 
parsing, or summarizing, the data in a reasonable way.

Since the Alexander polynomial can clearly be related to topology, it is  natural to want to relate the
new invariants to topological or geometric features of knots or links. 
In the 1980s William Thurston's seminal result
\cite[Corollary 2.5]{Thurston82}
 that most knot complements have
the structure of a hyperbolic manifold,
combined with Mostow's rigidity theorem 
\cite[Theorem 3.1]{Thurston82}
giving uniqueness of such structures,
establishes a strong connection between hyperbolic geometry and knot theory, since knots are determined by their complements.
Indeed, any geometric invariant of a knot complement, such as the hyperbolic volume, becomes a topological invariant of the knot.
Thus,  investigating if data derived from the new knot invariants is related to natural differential 
geometric invariants becomes another natural problem.

In this paper, we take up the family of weaving 
knots $W(p,n)$, where $p$ and $n$ are positive integers. We compute the 
signature for the general knot $W(p,n)$, and compute the polynomial knot invariants  for the subfamily $W(3,n)$.  
We explore the two problems just mentioned using our computations.  
In particular, Dasbach and Lin  \cite{DasbachLin} have provided some bounds on the hyperbolic volume of alternating knots 
in terms of coefficients of the Jones polynomial. They also define {\it higher twist numbers} for knots in terms of coefficients
of the Jones polynomials and suggested that these may have some correlation with the hyperbolic volume of the knots. 
We investigate this idea for the $W(3,n)$ knots.  As for parsing the enormous amount of numerical information
yielded by our methods, we explore the approximation of normalized Khovanov homology by normal distributions
from mathematical statistics.  A preprint   \cite{Distributions} developing this idea further is in preparation. 

Let us pause for more explanation of our decision to focus on weaving knots. 
According to \cite{Weaving_vol}, these knots have recently attracted interest, 
because it was conjectured that their complements would have the largest hyperbolic volume for a fixed crossing number.  
Concerning the conjecture about the volume, we cite the following theorem.  
\begin{theorem*}
  [Theorem 1.1, \cite{Weaving_vol}] 
If $p \geq 3$ and $q \geq 7$, then
\begin{equation} \label{CKPbounds}
  v_{{\rm oct}}(p-2)\,q\,\biggl(1 - \frac{(2\pi)^2}{q^2}\biggr)^{3/2} \leq {\rm vol}(W(p,q)) 
             < \bigl(v_{{\rm oct}}(p-3) + 4\,v_{{\rm tet}}) q.
\end{equation}
\end{theorem*}
Here $v_{{\rm oct}}$ and $v_{{\rm tet}}$ denote the volumes of the ideal octahedron and ideal tetrahedron respectively.
Champanerkar, Kofman, and Purcell call these bounds asymptotically sharp because their ratio approaches 1, 
as $p$ and $q$ tend to infinity.  Since the crossing number of $W(p,q)$ is known to be $(p{-}1)q$, the volume 
bounds in the theorem imply
\begin{equation*}
  \lim_{p,q \to \infty}\frac{{\rm vol}(W(p,q))}{c(W(p,q))} = v_{{\rm oct}} \approx 3.66
\end{equation*}
Their study raises the general question of examining the asymptotic behaviour of  other
invariants of weaving knots.  We also investigate the efficiency of the upper and lower bounds given in the theorem
for weaving knots $W(3,n)$.  

Turning to practical matters, here is the weaving knot $W(3,4)$.
\begin{figure}[h!]  %W(3,4)
  \begin{equation*}
  \xygraph{  !{0;/r1.5pc/:}
!{\hcap}[u]
!{\hcap[3]}[u]
!{\hcap[5]}[llllllll]
!{\xcaph[-8]@(0)}[dl]
!{\xcaph[-8]@(0)}[dl]
!{\xcaph[-8]@(0)}[uul]
!{\hcap[-5]}[d]
!{\hcap[-3]}[d]
!{\hcap[-1]}[d]
!{\xcaph[1]@(0)}[dl]
!{\htwist}[d]
!{\xcaph[1]@(0)}[uul]
!{\htwistneg}
!{\xcaph[1]@(0)}[dl]
!{\htwist}[d]
!{\xcaph[1]@(0)}[uul]
!{\htwistneg}
!{\xcaph[1]@(0)}[dl]
!{\htwist}[d]
!{\xcaph[1]@(0)}[uul]
!{\htwistneg}
!{\xcaph[1]@(0)}[dl]
!{\htwist}[d]
!{\xcaph[1]@(0)}[uul]
!{\htwistneg}
}    
  \end{equation*} \vspace{-8em}
\end{figure} 
%The purpose of the \vspace command is solely to remove excess space between
%the end of the figure and the next line of text.
\newline 
Enumerating strands $1, \ldots, p$ from the outside inward, our example is
the closure of the braid $(\sigma_1 \sigma_2^{-1})^4$ on three strands.
The weaving knot $W(p,n)$ is obtained from the torus knot $T(p,n)$
by making a standard diagram of the torus knot alternating. Symbolically, $T(p,n)$ is the
closure of the braid $(\sigma_1 \sigma_2 \cdots \sigma_{p-1})^n$, and
$W(p,n)$ is the closure of the braid $(\sigma_1 \sigma_2^{-1} \cdots \sigma_{p-1}^{\pm 1})^n$.
Obviously, the parity of $p$ is important.  If the greatest common divisor $\gcd(p,q) > 1$, 
then $T(p,n)$ and $W(p,n)$ are both links with $\gcd(p,n)$ components.  In general we
are interested only in the cases when $W(p,n)$ is an actual knot.  
The first invariant that we compute for $W(p,n)$ is the signature $\sigma(W(p,n))$ 
using a combinatorial method useful for alternating knot diagrams discussed 
in \cite{Lee_Endo04}.  We then  focus on the knots $W(3,n)$ which are closures of $3$-strand braids. 

From early in the development of Khovanov homology, computer experimentation has played a role in advancing the theory.
For example, \cite{BarNatan_Categorification} provided {\em Mathematica} routines to
compute Khovanov homology of knots of up to 10 crossings and provided tables of Betti numbers.  Based on the computations,
he makes a number of observations and conjectures about the structure of Khovanov homology. 
Subsequently, \cite{Khovanov_Patterns} recorded some observations about patterns in Khovanov homology and a remarkable 
relationship between the volume of a knot complement and the total rank of the knot's Khovanov homology. Some of the
conjectures on patterns are proved in \cite{Lee_Endo04}, on which our results depend.
In this paper we start a study of the asymptotic behavior of Khovanov homology of weaving knots.  With the assistance
of the computer algebra system {\em Maple}, we provide data on
the Khovanov homology of weaving knots $W(3,n)$ with up to 652 crossings.  The Khovanov homology groups are truly enormous,
so we approximate the distribution of dimensions using probability  density functions. We find that
normal distributions fit the data surprisingly well.

In more detail, this paper is organized as follows.
In section \ref{Weaving}, we discuss the generalities of weave knots and compute the signature $\sigma(W(p,n))$. 
We also make observation on Rasmussen's invariant \cite{Rasmussen} for these knots.
In section \ref{Hecke} we prepare to  follow the development of polynomial invariants in \cite{Jones_poly86}, starting from 
representations of braid groups into Hecke algebras.  For weaving knots $W(3,n)$, which are naturally represented as the closures 
of braids on three strands, we develop recursive formulas for their representations in the Hecke algebras. 
We note that the  survey \cite{BirmanMenasco} collects a number of facts and tools to facilitate computations
of invariants of knots and links that are the closures of 3-strand braids, as well as classifying the prime knots that are 
closure of 3-strand braids but not of 2-strand braids.  We would also like to point out that 
\cite{Takahashi} studies the Jones polynomials of knots that are the closures of general 3-strand braids, 
but the method is based on the  skein relation satisfied by the Jones polynomial. 
Our formulas are used not only 
in computer calculations of the polynomial invariants we need, but also in the development of information about particular
coefficients of these polynomials.  

Section \ref{Polys} builds on the recursion formulas 
to develop information about the Alexander polynomial $\Delta_{W(3,n)}(t)$ and Jones polynomial $V_{W(3,n)}(t)$.
As an application we exploit the relation between the Heegard-Floer homology associated to an alternating knot
and its Alexander polynomial  given in \cite[Theorem 1.3]{OSFloer},
as well as properties of Heegard-Floer homology,
 to prove that the complements of the knots $W(3,n)$ are fibered over $S^1$. 

In section \ref{TwistNumbersVolume} we investigate for knots $W(3,n)$ the higher twist numbers defined by Dasbach and Lin 
in \cite{DasbachLin} in terms of the Jones polynomial.  They ask if there is a correlation between the values of the higher
twist numbers and the hyperbolic volume of the knot complements.   As new results we have formulas for the second and 
third twist numbers of the knots $W(3,n)$, as well as conjectures for even higher twist numbers.  We believe that
improved lower bounds on the volume of knots $W(3,n)$ can be derived from the higher twist numbers, displaying 
results of computer experiments to support this idea.   We also exhibit plots of higher twist numbers against volume
that support the idea that values of higher twist numbers are correlated with volume.  

 In section \ref{Jones-to-Khovanov} we explain how to obtain the two-variable Poincar\'{e} polynomial for Khovanov homology 
with rational coefficients, and 
we present the results of calculations in a few relatively small examples.  
Using recent results of Shumakovitch \cite{Khovanovtorsion} that explain how, for alternating knots, the rational  Khovanov homology
determines the integral Khovanov homology, we provide a display of the integral Khovanov homology of the knot $W(3,4)$.
Concerning rational Khovanov homology, we observe that the distributions of 
dimensions in Khovanov homology resemble normal distributions.  We explore this further in section \ref{Data}, where we present
tables displaying summaries of calculations for weaving knots $W(3,n)$ for selected values of $n$ satisfying $\gcd(3, n) = 1$ 
and ranging up to $n=326$.  The standard deviation $\sigma$ of the normal distribution we attach to the Khovanov homology
of a weaving knot is a significant parameter.  The geometric significance of this number is an open question. 

In section \ref{PolynomialsExtra} we display expressions for a few polynomials arising from the Hecke algebra representations
of braid representations of $W(3,n)$, as they are used in section \ref{TwistNumbersVolume}, and values of the Jones polynomial,
Alexander polynomial and HOMFLY-PT polynomial for knots $W(3,4)$, $W(3,5)$, $W(3,10)$ and $W(3, 11)$. Finally, in section \ref{ComputerNotes}
we provide some information about the computer experiments we have performed with the knots $W(3,n)$ and how the experimental
results influenced the formulations of propositions and theorems. 

Finally, we thank several colleagues, especially Dr.~Joan Birman, for advice and suggestions concerning the exposition.
\section{Generalities on Weaving knots} \label{Weaving}
We have already mentioned that weaving knots are alternating by definition.
Various facts about alternating knots facilitate our calculations of the Khovanov homology of weaving knots $W(3,n)$.
For example, we appeal first to the following theorem of Lee.
\begin{theorem}[Theorem 1.2, \cite{Lee_Endo04}]  \label{locateKHLee}
For any alternating knot $L$ the Khovanov invariants
${\mathcal H}^{i,j}(L)$ 
are supported in two lines
\begin{equation*}
  j = 2i -\sigma(L) \pm 1.  \qed
\end{equation*}
 \end{theorem}
We will see that this result also has several practical implications.   
For example, to obtain a vanishing result for a particular alternating
knot, it suffices to compute the signature.  
Likewise, in connection with Heegard-Floer homology for $W(p,n)$,
\cite[Theorem 1.3]{OSFloer}
 essentially says that Heegard-Floer homology for an alternating knot is completely determined 
by the coefficients of its Alexander polynomial and the signature.   

Thus, it is important to compute the signature.
 Indeed, it turns out that there is a combinatorial formula for the
signature of oriented non-split alternating links. To state the formula requires the following terminology.
\begin{definition} \label{crossings}
  For a link diagram $D$ let $c(D)$ be the number of crossings of $D$, let $x(D)$ be number of
negative crossings, and let $y(D)$ be the number of positive crossings. For an
oriented link diagram, let $o(D)$ be the number
of components of $D(\emptyset)$, the diagram obtained by $A$-smoothing every crossing.
\end{definition}
\vspace{-3em}
\begin{figure}[h!]
  \begin{minipage}[h!]{0.5\linewidth} 
  \centering
\begin{equation*}
\UseComputerModernTips  \xygraph{ !{0;/r3pc/:} 
!{\htwist=>}[rr] 
!{\htwistneg=>}[rr]
}
\end{equation*}   
\caption{Positive and negative crossings}
  \end{minipage}
 \begin{minipage}[h!]{0.5\linewidth}
  \centering
\begin{equation*}
\UseComputerModernTips  \xygraph{ !{0;/r3pc/:} 
!{\huntwist}[rr]
!{\vuntwist}
}
\end{equation*}   
\caption{$A$-smoothing a positive, resp., negative, crossing}
  \end{minipage}
   \label{fig:crossings_smoothings}
\end{figure}
In words, $A$-regions in a neighborhood of a crossing are the regions swept out as the upper strand sweeps
counter-clockwise toward the lower strand.  An $A$-smoothing removes the crossing to connect these regions.
With these definitions,  we may cite the following proposition.
\begin{proposition}[Proposition 3.11, \cite{Lee_Endo04}]  
\label{basic_signature} 
For an oriented non-split alternating link $L$ and a reduced alternating diagram $D$ of $L$, 
$\sigma(L) = o(D) - y(D) -1$. \qed
\end{proposition}
We now use this result to compute the signatures of weaving knots. 
For a knot or link $W(m, n)$ drawn in the usual way,  the number of crossings $c(D) = (m{-}1)n$.
In particular, for $W(2k{+}1,n)$, $c\bigl(W(2k{+}1,n)\bigr)= 2kn$;   for $W(2k, n)$, $c\bigl( W(2k, n) \bigr) = (2k{-}1)n$. 
Evaluating the other quantities in definition \ref{crossings}, we calculate the signatures of weaving knots. 
\begin{proposition} \label{weavingsignature}
  For a weaving knot $W(2k{+}1,n)$, $\sigma\bigl( W(2k{+}1,n) \bigr) = 0$,  
and for $W(2k, n)$,   $\sigma\bigl( W(2k, n) \bigr) =  -n{+}1$.
\end{proposition}
\begin{corollary}
  For a weaving knot $W(2k{+}1, n)$, the Rasmussen $s$-invariant is zero.  For
  a weaving knot $W(2k, n)$, the Rasmussen $s$-invariant is $-n{+}1$.
\end{corollary}
\begin{proof}
  For alternating knots, it is known that the $s$-invariant coincides with the
  signature \cite[Theorems 1--4]{Rasmussen}.
\end{proof}
\begin{figure}[h!]
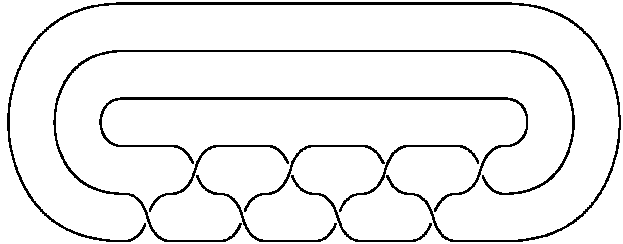

  \centering
   %W(3,4)
\begin{equation*}
\xygraph{  !{0;/r1.5pc/:}
!{\hcap}[u]
!{\hcap[3]}[u]
!{\hcap[5]}[llllllll]
!{\xcaph[-8]@(0)}[dl]
!{\xcaph[-8]@(0)}[dl]
!{\xcaph[-8]@(0)}[uul]
!{\hcap[-5]}[d]
!{\hcap[-3]}[d]
!{\hcap[-1]}[d]
!{\xcaph[1]@(0)}[dl]
!{\htwist}[d]
!{\xcaph[1]@(0)}[uul]
!{\htwistneg}
!{\xcaph[1]@(0)}[dl]
!{\htwist}[d]
!{\xcaph[1]@(0)}[uul]
!{\htwistneg}
!{\xcaph[1]@(0)}[dl]
!{\htwist}[d]
!{\xcaph[1]@(0)}[uul]
!{\htwistneg}
!{\xcaph[1]@(0)}[dl]
!{\htwist}[d]
!{\xcaph[1]@(0)}[uul]
!{\htwistneg}
}   
\end{equation*}   \vspace{-9em}
  \caption{The weaving knot $W(3,4)$}
  \label{fig:w34}
\end{figure}
\begin{figure}[h!]
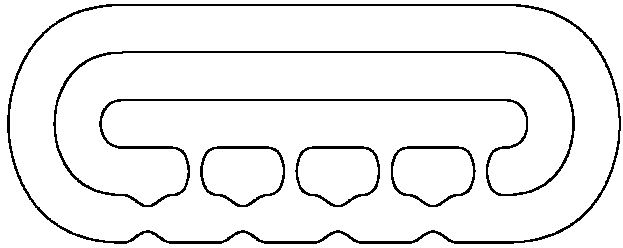

  \centering
\begin{equation*}  %A-smoothing W(3,4)
  \xygraph{  !{0;/r1.5pc/:}
!{\hcap}[u]
!{\hcap[3]}[u]
!{\hcap[5]}[llllllll]
!{\xcaph[-8]@(0)}[dl]
!{\xcaph[-8]@(0)}[dl]
!{\xcaph[-8]@(0)}[uul]
!{\hcap[-5]}[d]
!{\hcap[-3]}[d]
!{\hcap[-1]}[d]
!{\xcaph[1]@(0)}[dl]
!{\huntwist}[u]
!{\huncross}[ddl]
!{\xcaph[1]@(0)}[uu]
!{\xcaph[1]@(0)}[dl]
!{\huntwist}[u]
!{\huncross}[ddl]
!{\xcaph[1]@(0)}[uu]
!{\xcaph[1]@(0)}[dl]
!{\huntwist}[u]
!{\huncross}[ddl]
!{\xcaph[1]@(0)}[uu]
!{\xcaph[1]@(0)}[dl]
!{\huntwist}[u]
!{\huncross}[ddl]
!{\xcaph[1]@(0)}[uu]
}
\end{equation*}  \vspace{-9em}
  \caption{The $A$-smoothing of $W(3,4)$}
  \label{fig:w34a}
\end{figure}
\begin{proof}[Proof of Proposition \ref{weavingsignature}]
  Consider first the example $W(3,n)$, illustrated by figures \ref{fig:w34} and \ref{fig:w34a} for $W(3,4)$.
After $A$-smoothing the diagram, the outer ring of crossings produces a circle bounding the 
rest of the smoothed diagram.
On the inner ring of crossings the $A$-smoothings produce $n$ circles in a cyclic arrangement.  Therefore
$o\bigl( W(3,n) \bigr) = 1 + n$.  The outer ring of crossings consists of positive crossings
and the inner ring of crossings consists of negative crossings, so $x(D) = y(D) = n$.  
Applying the formula of proposition \ref{basic_signature}, we obtain the result $\sigma\bigl( W(3,n) \bigr)=0$.  

For the general case $W(2k{+}1, n)$, we have the following considerations.
The crossings are organized into  $2k$ rings.  Reading from the outside toward the center, we have
first a ring of positive crossings, then a ring of negative crossings, and so on, alternating
positive and negative.  Thus, $y(D) = kn$.  Considering the $A$-smoothing of the diagram of $W(2k{+}1,n)$,
as in the special case, a bounding circle appears from the smoothing of the outer ring.  
A chain of $n$ disjoint smaller circles appears inside the second ring.  No circles appear in the third ring, nor in any odd-numbered ring 
thereafter.  On the other hand, chains of $n$ disjoint smaller circles appear in each even-numbered ring.  Since there are $k$ even-numbered
rings, we have $o(D) = 1 + kn$. 
Applying the formula of proposition \ref{basic_signature}
\begin{equation*}
  \sigma\bigl( W(2k{+}1, n)\bigr)  = o(D) - y(D) -1 = (1+kn) - kn -1 = 0.
\end{equation*}
\begin{figure}[h!]
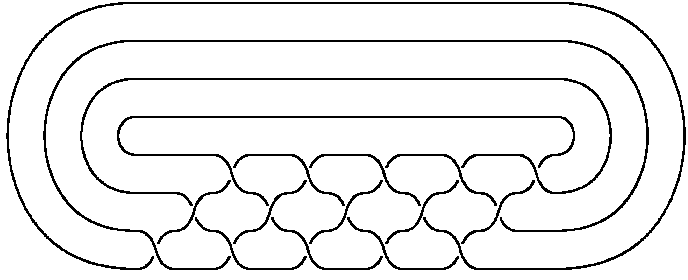

  \centering
\begin{equation*}
   \xygraph{  !{0;/r1.2pc/:}
!{\hcap}[u]
!{\hcap[3]}[u]
!{\hcap[5]}[u]
!{\hcap[7]}[lllllllllll]
!{\xcaph[-11]@(0)}[dl]
!{\xcaph[-11]@(0)}[dl]
!{\xcaph[-11]@(0)}[dl]
!{\xcaph[-11]@(0)}[uuul]
!{\hcap[-7]}[d]
!{\hcap[-5]}[d]
!{\hcap[-3]}[d]
!{\hcap[-1]}[d]
!{\xcaph[2]@(0)}[dl]
!{\xcaph[1]@(0)}[dl]
!{\htwist}[d]
!{\xcaph[1]@(0)}[uul]
!{\htwistneg}[u]
!{\htwist}[ddl]
!{\htwist}[d]
!{\xcaph[1]@(0)}[uul]
!{\htwistneg}[ul]
!{\xcaph[1]@(0)}
!{\htwist}[ddl]
!{\htwist}[d]
!{\xcaph[1]@(0)}[uul]
!{\htwistneg}[ul]
!{\xcaph[1]@(0)}
!{\htwist}[ddl]
!{\htwist}[d]
!{\xcaph[1]@(0)}[uul]
!{\htwistneg}[ul]
!{\xcaph[1]@(0)}
!{\htwist}[ddl]
!{\htwist}[d]
!{\xcaph[2]@(0)}[uul]
!{\htwistneg}[ul]
!{\xcaph[1]@(0)}
!{\htwist}[ddl]
!{\xcaph[1]@(0)}
}
\end{equation*} \vspace{-9em}
  \caption{The weaving knot $W(4,5)$}
  \label{fig:w45}
\end{figure}
\begin{figure}[h!]
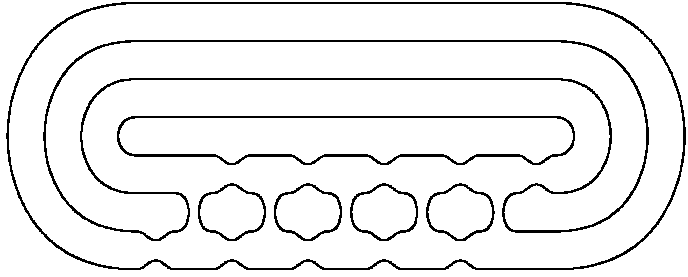

  \centering
\begin{equation*}  %A smoothing of knot $W(4,5)$
  \xygraph{ !{0;/r1.2pc/:} 
!{\hcap}[u]
!{\hcap[3]}[u]
!{\hcap[5]}[u]
!{\hcap[7]}[lllllllllll]
!{\xcaph[-11]@(0)}[dl]
!{\xcaph[-11]@(0)}[dl]
!{\xcaph[-11]@(0)}[dl]
!{\xcaph[-11]@(0)}[uuul]
!{\hcap[-7]}[d]
!{\hcap[-5]}[d]
!{\hcap[-3]}[d]
!{\hcap[-1]}[d]
!{\xcaph[2]@(0)}[dl]
!{\xcaph[1]@(0)}[dl]
!{\huntwist}[d]
!{\xcaph[1]@(0)}[uul]
!{\huncross}[u]
!{\huntwist}[ddl]
!{\huntwist}[d]
!{\xcaph[1]@(0)}[uul]
!{\huncross}[ul]
!{\xcaph[1]@(0)}
!{\huntwist}[ddl]
!{\huntwist}[d]
!{\xcaph[1]@(0)}[uul]
!{\huncross}[ul]
!{\xcaph[1]@(0)}
!{\huntwist}[ddl]
!{\huntwist}[d]
!{\xcaph[1]@(0)}[uul]
!{\huncross}[ul]
!{\xcaph[1]@(0)}
!{\huntwist}[ddl]
!{\huntwist}[d]
!{\xcaph[2]@(0)}[uul]
!{\huncross}[ul]
!{\xcaph[1]@(0)}
!{\huntwist}[ddl]
!{\xcaph[1]@(0)}
}
\end{equation*} \vspace{-9em}  
  \caption{The $A$-smoothing of $W(4,5)$}
  \label{fig:w45a}
\end{figure}
For the case $W(2k, n)$,  we show  $W(4,5)$  in figures \ref{fig:w45} and \ref{fig:w45a} as an example. 
Our standard diagram may be organized into $2k{-}1$ rings of crossings.  
In each ring there are $n$ crossings, so the total number of crossings is $c(D) = (2k{-}1)n$. 
In our standard representation, there is an outer ring of $n$ positive crossings, next a ring of $n$ negative crossings, alternating 
until we end with an innermost ring of $n$ positive crossings.  There are thus $k$ rings of $n$ positive crossings and $k{-}1$ rings
of $n$ negative crossings.  Therefore,  $y(D) = kn$  and  $x(D) = (k{-}1)n$.
 Considering the $A$-smoothing of the diagram, a bounding circle appears from the smoothing of the outer ring.  As before, a chain of $n$ disjoint
smaller circles appears inside the second ring and in each successive even-numbered ring.  As previously noted, there are $k{-}1$ of these
rings.  No circles appear in odd-numbered rings, until we reach the last ring, where an inner bounding circle appears. 
Thus, $o(D) = 1 + (k{-}1)n + 1 = (k{-}1)n + 2$. 
Consequently, 
\begin{equation*}
    \sigma\bigl( W(2k, n)\bigr)  = o(D) - y(D) -1 = \bigl((k{-}1)n + 2\bigr) - kn - 1
                    = -n{+}1.  \qedhere
\end{equation*}
\end{proof}
\begin{theorem}
  \label{locateKH}
For a weaving knot $W(2k{+}1,n)$ the non-vanishing Khovanov homology ${\mathcal H}^{i,j}\bigl( W(2k{+}1, n) \bigr)$ lies on the 
lines
\begin{equation*}
  j = 2i \pm 1.
\end{equation*}
For a weaving knot $W(2k, n)$ the non-vanishing Khovanov homology ${\mathcal H}^{i,j}\bigl( W(2k, n) \bigr)$ lies on the lines
\begin{equation*}
  j = 2i + n -1 \pm 1
\end{equation*}
\end{theorem}
\begin{proof}
  Substitute the calculations made in proposition \ref{weavingsignature} into
  the formula of theorem \ref{locateKHLee}.
\end{proof}
\section{Recursion in the Hecke algebra} \label{Hecke}
We review briefly the definition of the Hecke algebra $H_{N+1}$ on generators $1$ and $T_1$ through $T_N$,
and we define the representation of the braid group $B_3$ on three strands in $H_3$. 
Theorem \ref{heckerecursion} sets up recursion relations for the coefficients in the expansion 
of the image in $H_3$ of the braid $(\sigma_1 \sigma_2^{-1})^{n}$,
whose closure is the weaving knot $W(3,n)$.
These coefficients are polynomials in a parameter $q$, which is built into the definition \ref{Heckealgebras} of
the Hecke algebra.  The recursion relations are  essential for automating 
the calculation of the Jones polynomial for the knots $W(3,n)$.  
Proposition \ref{C121} uses the relations developed in theorem \ref{heckerecursion}
to prove a vanishing result for one of the coefficients.  Being able to ignore one of 
the coefficients speeds up the computations slightly.
 Proposition \ref{trailing} evaluates the constant terms of the families of polynomials,
proposition \ref{degreeone} evaluates the degree one coefficients, and 
theorem \ref{palindromes} proves certain identities satisfied by the polynomials.
These identities imply symmetry properties of the coefficients and enable calculation 
of the degrees of the polynomials in corollary \ref{degrees1}.
\begin{definition}
  \label{Heckealgebras}
Working  over the ground field $K$ containing an element $q \neq 0$, the Hecke
algebra $H_{N+1}$ is the associative algebra with $1$ on generators $T_1$, \ldots, $T_N$ satisfying these relations.
\begin{align}
  T_iT_j &= T_jT_i, \quad \text{whenever $\abs{i-j} \geq 2$,}  \label{commutativity}
\\
  T_iT_{i+1}T_i &= T_{i+1}T_iT_{i+1}, \quad \text{for $1 \leq i \leq N{-}1$,} \label{interchange} 
\intertext{and,  finally,}
  T_i^2 &= (q{-}1)T_i + q,\quad  \text{for all $i$.} \label{inverse}
\end{align}
It is well-known 
\cite{Jones_poly86}
that  $(N{+}1)!$ is the dimension of $H_{N+1}$ over $K$.
\end{definition}
Recasting the relation 
$T_i^2 = (q{-}1)T_i + q$ in the form $q^{-1} \bigl(T_i - (q-1)\bigr)\cdot T_i = 1$ shows that 
$T_i$ is invertible in $H_{N+1}$ with 
$T_i^{-1} = q^{-1}\bigl( T_i - (q-1) \bigr)$.
Consequently, the specification
$\rho(\sigma_i)  = T_i$, combined with relations \eqref{commutativity} and \eqref{interchange}, defines  a homomorphism 
$\rho \colon B_{N+1} \ra H_{N+1} $ from $B_{N+1}$, the group of braids on $N{+}1$ strands,
 into the multiplicative monoid of $H_{N+1}$. 

For work in $H_3$, choose the  ordered basis
$\{1, T_1, T_2, T_1T_2, T_2T_1, T_1T_2T_1\}$.
The word in the Hecke algebra corresponding to the knot $W(3,n)$ is formally
\begin{equation}
  \label{eq:basicw3n}
 \rho\bigl( (T_1T_2^{-1})^n\bigr)  
        = q^{-n}\bigl(C_{n,0}+ C_{n,1}\cdot T_1 + C_{n,2} \cdot T_2 + C_{n,12} \cdot T_1T_2  
               + C_{n,21} \cdot T_2T_1  + C_{n,121} \cdot T_1T_2T_1 \bigr),
\end{equation}
where the coefficients  $C_{n,*} = C_{n,*}(q)$ of the monomials in $T_1$ and $T_2$ are  polynomials in $q$.
For $n = 1$, 
  \begin{equation*}
  \rho ( \sigma_1 \sigma_2^{-1}) = T_1T_2^{-1} = q^{-1}\cdot\bigl(  T_1 ( -(q{-}1) + T_2)  \bigr) 
  = q^{-1}\bigl( -(q{-}1)\cdot T_1 + T_1T_2 \bigr),
  \end{equation*}
so we have initial values 
\begin{equation}  \label{initialCs}
  C_{1,0}(q) = 0, \; C_{1,1}(q) = -(q{-}1), \; C_{1,2}(q) = 0, \; C_{1,12}(q) = 1,
                   \;  C_{1,21}(q) = 0, \; \text{and} \; C_{1,121}(q) = 0.
\end{equation}
\begin{theorem}
  \label{heckerecursion}
These polynomials satisfy the following recursion relations.
\begin{align}
C_{n,0}(q) &= q^2\cdot C_{n-1,21}(q) - q(q{-}1)\cdot C_{n-1,1}(q)  \label{cn0}
\\
  C_{n,1}(q) &= - (q{-}1)^2\cdot C_{n-1,1}(q) - (q{-}1)\cdot C_{n-1,0}(q) + q^2\cdot C_{n-1,121}(q) \label{cn1}
\\
  C_{n,2}(q) &=  q\cdot C_{n-1,1}(q)  \label{cn2}
\\
 C_{n,12}(q)  &= (q{-}1)\cdot C_{n-1,1}(q) + C_{n-1,0}(q)  \label{cn12}
\\
\begin{split}
  C_{n, 21}(q) &= -(q{-}1)\cdot C_{(n-1),2}(q) + q\cdot C_{n-1,12}(q)\\ 
  & \hspace{3em} - (q{-}1)^2\cdot C_{n-1,21}(q)  + q(q{-}1)\cdot C_{n-1, 121}(q) \label{cn21}
\end{split}
\\
C_{n,121}(q) &= C_{n-1,2}(q) + (q{-}1)\cdot C_{n-1,21}(q)  \label{cn121}
\end{align}
\end{theorem}
\begin{proof}
  We have
  \begin{multline} \label{exp0}
    \rho( T_1T_2^{-1} )^n =  \rho( T_1T_2^{-1} )^{n-1} \cdot \rho (T_1T_2^{-1}) 
\\
     =   q^{-n+1}\bigl(C_{n-1,0}+ C_{n-1,1}\cdot T_1 + C_{n-1,2} \cdot T_2 + C_{n-1,12} \cdot T_1T_2  
               + C_{n-1,21} \cdot T_2T_1  + C_{n-1,121}\cdot T_1T_2T_1 \bigr)
                 \\  \cdot q^{-1}\bigl( -(q{-}1)\cdot T_1 + T_1T_2 \bigr)
\\
    =  q^{-n}\biggl( -(q{-}1)C_{n-1,0}\cdot T_1  -(q{-}1)C_{n-1,1}\cdot T_1^2   -(q{-}1)C_{n-1,2} \cdot T_2T_1 
\\
                -(q{-}1)C_{n-1,12} \cdot T_1T_2T_1   -(q{-}1)C_{n-1,21} \cdot T_2T_1^2   -(q{-}1)C_{n-1,121}\cdot T_1T_2T_1^2
\\
          + C_{n-1,0} \cdot T_1T_2+ C_{n-1,1}\cdot T_1^2T_2 + C_{n-1,2} \cdot T_2T_1T_2
\\
           + C_{n-1,12} \cdot T_1T_2T_1T_2   + C_{n-1,21} \cdot T_2T_1^2T_2  + C_{n-1,121}\cdot T_1T_2T_1^2T_2 \biggr)
\\ 
   = q^{-n}\biggl( \Bigl( -(q{-}1)C_{n-1,0}\cdot T_1  -(q{-}1)C_{n-1,2} \cdot T_2T_1  -(q{-}1)C_{n-1,12} \cdot T_1T_2T_1 +  C_{n-1,0} \cdot T_1T_2\Bigr)
\\
  + \Bigl\{-(q{-}1)C_{n-1,1}\cdot T_1^2 -(q{-}1)C_{n-1,21} \cdot T_2T_1^2 -(q{-}1)C_{n-1,121}\cdot T_1T_2T_1^2 +  C_{n-1,1}\cdot T_1^2T_2
\\
     +C_{n-1,2} \cdot T_2T_1T_2  + C_{n-1,12} \cdot T_1T_2T_1T_2   + C_{n-1,21} \cdot T_2T_1^2T_2  + C_{n-1,121}\cdot T_1T_2T_1^2T_2 \Bigr\} \biggr)
  \end{multline}
after collecting powers of $q$ and expanding.  In the last grouping, the first four terms inside the parentheses $( \ )$ 
involve only elements of the preferred basis;
the second eight terms in the pair of braces $\{\ \}$ all require further expansion, as follows. 
\begin{align}
 \begin{split}  -(q{-}1)C_{n-1,1}\cdot T_1^2 &= -(q{-}1)C_{n-1,1}\cdot ((q-1)T_1 + q) 
\\   &= -(q-1)^2C_{n-1,1}\cdot T_1 - q(q-1)C_{n-1,1} \end{split} \label{expa}
\\
 \begin{split}  -(q{-}1)C_{n-1,21} \cdot T_2T_1^2 &= -(q{-}1)C_{n-1,21} \cdot T_2((q-1)T_1 + q) 
\\    &= -(q-1)^2C_{n-1,21}\cdot T_2T_1 - q(q-1)C_{n-1,21} \cdot T_2  \end{split} \label{expb}
\\
\begin{split} -(q{-}1)C_{n-1,121}\cdot T_1T_2T_1^2 &= -(q{-}1)C_{n-1,121} \cdot T_1T_2((q-1)T_1 + q) 
\\    &   = -(q-1)^2C_{n-1,121}\cdot T_1T_2T_1 - q(q-1)C_{n-1,121} \cdot T_1T_2 \end{split} \label{expc}
\\
\begin{split} C_{n-1,1}\cdot T_1^2T_2 &= C_{n-1,1}\cdot( (q{-}1) T_1 + q) T_2 
\\    &= (q{-}1)C_{n-1,1}\cdot T_1T_2 + qC_{n-1,1} \cdot T_2
\end{split} \label{expd}
\\
 C_{n-1,2} \cdot T_2T_1T_2 &= C_{n-1,2} \cdot T_1T_2T_1 \label{expe}
\\
\begin{split}   
  C_{n-1,12} \cdot T_1T_2T_1T_2   &=  C_{n-1,12}\cdot T_1^2T_2T_1 = C_{n-1,12}((q{-}1)T_1 + q)T_2T_1
\\ &= (q{-}1)C_{n-1,12}\cdot T_1T_2T_1 + qC_{n-1,12}\cdot T_2T_1
\end{split} \label{expf}
\\
\begin{split}
  C_{n-1,21} \cdot T_2T_1^2T_2  &= C_{n-1,21}\cdot T_2 ((q{-}1)T_1 + q) T_2
\\ &= (q{-}1)C_{n-1,21} \cdot T_2T_1T_2 + qC_{n-1,21} \cdot T_2^2
\\ &= (q{-}1)C_{n-1,21} \cdot T_1T_2T_1 + qC_{n-1,21}\cdot ((q{-}1)T_2 + q) 
\\ &=  (q{-}1)C_{n-1,21} \cdot T_1T_2T_1 + q(q{-}1)C_{n-1,21}\cdot T_2 + q^2C_{n-1,21})
\end{split} \label{expg}
\\
\begin{split}
 C_{n-1,121}\cdot T_1T_2T_1^2T_2  &=  C_{n-1,121} \cdot T_1T_2((q{-}1)T_1 + q)T_2 
\\ &= (q{-}1)C_{n-1,121} \cdot T_1T_2T_1T_2 + qC_{n-1,121} \cdot T_1T_2^2
 \\ &= (q{-}1)C_{n-1,121} \cdot T_1^2T_2T_1 + qC_{n-1,121}\cdot T_1((q{-}1)T_2+ q)
\\ &= (q{-}1)C_{n-1,121} \cdot ((q{-}1)T_1+ q)T_2T_1
\\  & \hspace{3em} + qC_{n-1,121}\cdot T_1((q{-}1)T_2+ q)
%The purpose of the \hspace command is just to move the + so that it does not line up under =.
\\ &= (q{-}1)^2C_{n-1,121}\cdot T_1T_2T_1 + q (q{-}1)C_{n-1,121}\cdot T_2T_1 
\\  & \hspace{3em} + q(q{-}1)C_{n-1,121}\cdot T_1T_2+ q^2C_{n-1,121} \cdot T_1
%The purpose of the \hspace command is just to move the + so that it does not line up under =.
\end{split} \label{exph}
\end{align}
Collecting the constant terms from \eqref{expa} and \eqref{expg}, we get
\begin{align*}
  C_{n,0} &=  - q(q-1)C_{n-1,1} + q^2C_{n-1,21}.
\intertext{Collecting coefficients of $T_1$ from \eqref{exp0}, \eqref{expa}, \eqref{exph}, we get}
  C_{n,1} &=  -(q{-}1)C_{n-1,0} -(q-1)^2C_{n-1,1} + q^2C_{n-1,121}.
\intertext{Collecting coefficients of $T_2$ from \eqref{expb}, \eqref{expd}, and \eqref{expg}, we get}
  C_{n,2} &=  - q(q-1)C_{n-1,21}  +  qC_{n-1,1}  + q(q{-}1)C_{n-1,21} = qC_{n-1,1}.
\intertext{Collecting coefficients of $T_1T_2$ from \eqref{exp0}, \eqref{expc}, \eqref{expd}, and \eqref{exph}, we get}
  C_{n,12} &=  C_{n-1,0}  - q(q-1)C_{n-1,121} + (q{-}1)C_{n-1,1}  + q(q{-}1)C_{n-1,121} =  C_{n-1,0}  + (q{-}1)C_{n-1,1}
\intertext{Collecting coefficients of $T_2T_1$ from \eqref{exp0}, \eqref{expb}, \eqref{expf}, and \eqref{exph}, we get}
  C_{n,21} &=  -(q{-}1)C_{n-1,2} -(q-1)^2C_{n-1,21} + qC_{n-1,12}  + q (q{-}1)C_{n-1,121}.
\intertext{Collecting coefficients of $T_1T_2T_1$ from \eqref{exp0}, \eqref{expc}, \eqref{expe}, \eqref{expf}, \eqref{expg}, and \eqref{exph},
we get}
 C_{n,121} &=   -(q{-}1)C_{n-1,12}  -(q-1)^2C_{n-1,121} + C_{n-1,2}  
\\ 
&\hspace{3em}+ (q{-}1)C_{n-1,12} +  (q{-}1)C_{n-1,21} + (q{-}1)^2C_{n-1,121}
%The purpose of the \hspace command is just to move the + so that it does not line up under =.
\\
         &= C_{n-1,2} +  (q{-}1)C_{n-1,21}
\end{align*}
Up to simple rearrangements and expansion of notation, these are formulas \eqref{cn0} through ~\eqref{cn121}.
\end{proof}
\begin{example} \label{secondCs}
Applying the recursion formulas just proved to the table of initial polynomials, 
or by computing $\rho\bigl( (\sigma_1 \sigma_2^{-1})^2 \bigr)$ directly from the definitions, 
we find
\begin{align}
  C_{2,0}(q) &= q^2 \cdot C_{1,21}(q) - q(q{-}1)\cdot C_{1,1}(q) = q(q{-}1)^2,
\\
  C_{2,1}(q) &= -(q{-}1)^2\cdot C_{1,1}(q)-(q{-}1)\cdot C_{1,0}(q) = (q{-}1)^3,
\\
  C_{2,2}(q) &= q \cdot C_{1,1}(q) = -q(q{-}1),
\\
  C_{2,12}(q) &= (q{-}1)\cdot C_{1,1}(q) + C_{1,0}(q) = -(q{-}1)^2,
\\
  C_{2,21}(q) &= -(q{-}1)\cdot C_{1,2}(q) + q \cdot C_{1,12}(q) - (q{-}1)^2\cdot C_{1,21}(q) = q,
\\
  C_{2,121}(q) &=0.
\end{align}
  \end{example}
As a first application, we have the following vanishing result.
\begin{proposition}
  \label{C121}
For all $n$, $C_{n,121}(q) = 0$.
\end{proposition}
\begin{proof}
For $n \geq 1$, we claim $C_{n+1,121}(q) = 0$.
Make the inductive assumption that $C_{k,121}(q) = 0$ for $1 \leq k \leq n$. Apply \eqref{cn121}, \eqref{cn21}, and the 
inductive hypothesis to write
\begin{multline*}
       C_{n+1,121}(q) = C_{n,2}(q) + (q{-}1)\cdot C_{n,21}(q)
\\
  \shoveleft = C_{n,2}(q) 
\\
   + (q{-}1)\Bigl( -(q{-}1)\cdot  C_{n-1,2}(q) + q\cdot C_{n-1,12}(q) - (q{-}1)^2\cdot C_{n-1,21}(q)
 + q(q{-}1)\cdot C_{n-1, 121} \Bigr) 
\\
  = C_{n,2}(q) + 
    (q{-}1)\bigl( -(q{-}1)\cdot C_{n-1,2}(q) + q\cdot C_{n-1,12}(q) -
    (q{-}1)^2\cdot C_{n-1,21}(q) \bigr).
\end{multline*}
Using \eqref{cn2} to replace the first term $C_{n,2}(q)$ and \eqref{cn12} to replace the third term factor $C_{n-1,12}(q)$ on the right,
\begin{align*}
  C_{n+1, 121}(q)  &= q\cdot C_{n-1,1}(q) - (q{-}1)^2C_{n-1,2}(q) + q(q{-}1)\bigl( (q{-}1)C_{n-2,1}(q)+C_{n-2,0}(q) \bigr)
 \\
 &\hspace{3em} - (q{-}1)^3C_{n,21}(q) 
%The purpose of the \hspace command is just to move the - so that it does not line up under =.
\\
  &= q\cdot C_{n-1,1}(q) - (q{-}1)^2C_{n-1,2}(q) + (q{-}1)^2\Bigl( q C_{n-2,1}(q) \Bigr) + q(q{-}1)C_{n-2,0}(q) 
\\ 
&\hspace{3em} - (q{-}1)^3C_{n,21}(q) 
%The purpose of the \hspace command is just to move the + so that it does not line up under =.
\\
  &= q\cdot C_{n-1,1}(q) - (q{-}1)^2C_{n-1,2}(q) + (q{-}1)^2C_{n-1,2}(q)+ q(q{-}1)C_{n-2,0}(q)
\\ 
&\hspace{3em} - (q{-}1)^3C_{n,21}(q),
%The purpose of the \hspace command is just to move the - so that it does not line up under =.
\end{align*}
where we use \eqref{cn2} in reverse to rewrite the term $q C_{n-2,1}(q) $.  Making the obvious cancellation,
\begin{align*}
  C_{n+1, 121}  &=  q\cdot C_{n-1,1}(q)+ q(q{-}1) \cdot C_{n-2,0}(q) - (q{-}1)^3 \cdot C_{n,21}(q) 
\\
&=q\bigl( C_{n-1,1}+ (q{-}1)C_{n-2,0}\bigr) - (q{-}1)^3 \cdot C_{n-1, 21}
\\
&= q\Bigl(\bigl(-(q{-}1)^2 \cdot C_{n-2,1} - (q{-}1) \cdot C_{n-2,0}\bigr) + (q{-}1) \cdot C_{n-2,0}\Bigr) - (q{-}1)^3 \cdot C_{n-1, 21},
\end{align*}
since 
\begin{align*}  
C_{n-1,1}(q) &= - (q{-}1)^2\cdot C_{n-2,1}(q) - (q{-}1)\cdot C_{n-2,0}(q) + q^2\cdot C_{n-2,121}(q)
\\
            & = - (q{-}1)^2\cdot C_{n-2,1}(q) - (q{-}1)\cdot C_{n-2,0}(q)
\end{align*}
by \eqref{cn1} and the inductive hypothesis.  Therefore, 
\begin{align*}
C_{n+1,121}(q)  &= -q(q{-}1)^2\cdot C_{n-2,1}(q)  - (q{-}1)^3\cdot C_{n-1, 21}(q)
\\
&= -(q{-}1)^2\cdot C_{n-1,2}(q) - (q{-}1)^3\cdot C_{n-1,21}(q),
\intertext{using \eqref{cn2} in the form $  C_{n-1,2}(q) =  q\cdot C_{n-2,1}(q) $,}
&= -(q{-}1)^2\bigl( C_{n-1, 2}(q) - (q{-}1)\cdot C_{n-1,21} (q)\bigr)
\\
&= -(q{-}1)^2\cdot C_{n, 121}(q) = 0,
\end{align*}
using \eqref{cn121} and the inductive hypothesis.
\end{proof}
\begin{proposition}
  \label{trailing}
For $n \geq 1$, the degree $0$ terms in the non-vanishing polynomials $C_{n,-}$ are as follows:
\begin{gather}
  C_{n,0}(0) = c_{n,0,0} = 0,\;C_{n,1}(0) = c_{n,1,0} = (-1)^{n-1},\; C_{n,2}(0) = c_{n,2,0}=0,\notag
\\
               C_{n,12}(0) =c_{n,12,0}= (-1)^{n-1},\; C_{n,21}(0) = c_{n,21,0}=0. \label{trailingcoefftable}
\end{gather}
\end{proposition}
Since proposition \ref{C121} proves that $C_{n, 121}(q)$ is identically zero, it does not appear in the list just given or anywhere
in the later parts of this paper.
\begin{proof}
  Examinination of  the formulas for the polynomials $C_{1,-}(q)$ given in equations \eqref{initialCs} 
and for $C_{2, -}(q)$ given in example \ref{secondCs} starts the inductive proof. 
Substituting $q{=}0$ in the recursive formulas \eqref{cn0} and \eqref{cn2} immediately yields
$C_{n,0}(0) = 0$ and $C_{n,2}(0) = 0$. Substituting $q{=}0$ into formula \eqref{cn1} yields
\begin{equation*}
  c_{n,1,0} = C_{n,1}(0) =  - (0{-}1)^2\cdot C_{n-1,1}(0) - (0{-}1)\cdot C_{n-1,0}(0) = -(-1)^{n-2} + 0 = (-1)^{n-1},
\end{equation*}
by the inductive hypothesis.  Similarly, substituting $q{=}0$ into formula \eqref{cn12} yields
\begin{equation*}
  c_{n,12,0} = C_{n,12}(0) =  (0{-}1)\cdot C_{n-1,1}(0) + C_{n-1,0}(0) = (-1)\cdot(-1)^{n-2} + 0 = (-1)^{n-1}.
\end{equation*}
Finally, substituting $q{=}0$ into formula \eqref{cn21} yields
\begin{multline*}
  c_{n,21,0}= C_{n,21}(0) = -(0{-}1)\cdot C_{n-1,2}(0) + 0\cdot C_{n-1,12}(0) - (0{-}1)^2\cdot C_{n-1,21}(0)
\\  
         = 1\cdot 0 + 0\cdot(-1)^{n-2} - 1\cdot 0 = 0. \qedhere
\end{multline*}
\end{proof}
\begin{proposition} \label{degreeone}
  For $n \geq 2$ the degree one coefficients in the non-vanishing polynomials $C_{n,-}(q)$  are as follows:
  \begin{gather}
    c_{n,0,1} = (-1)^n, \quad c_{n,1,1} = (-1)^n(n{+}1), \quad c_{n,2,1} = (-1)^n,  \notag
\\ 
     c_{n,12,1} = (-1)^nn, \quad  c_{n,21,1} = (-1)^n.  \label{degonetable}
  \end{gather}
%c_{n-1,1,1} = (-1)^{n-1}n    c_{n-1,12,1} = (-1)^{n-1}(n{-}1)
\end{proposition}
\begin{proof}
  These are all handled in the same manner.  Namely, differentiate the recursive relations
for each successive polynomial, substitute $q{=}0$, and use the values from proposition \ref{trailing} as appropriate.
Concerning $C_{n,0}(q)$, differentiate \eqref{cn0} with respect to $q$, obtaining
\begin{equation*}
  %C_{n,0}(q) &= q^2\cdot C_{n-1,21}(q) - q(q{-}1)\cdot C_{n-1,1}(q)
C_{n,0}'(q) = 2q \cdot  C_{n-1,21}(q) + q^2 C_{n-1,21}'(q)  - (2q{-}1)\cdot C_{n-1,1}(q) - q(q{-}1)\cdot C_{n-1,1}'(q),
\end{equation*}
whence
\begin{equation*}
  c_{n,0,1} = C_{n,0}'(0) = -(-1)\cdot C_{n-1,1}(0) = (-1)^{n-2}= (-1)^n.
\end{equation*}
Concerning $C_{n,1}(q)$, differentiate \eqref{cn1} with respect to $q$, obtaining
\begin{equation*}
  %  C_{n,1}(q) &= - (q{-}1)^2\cdot C_{n-1,1}(q) - (q{-}1)\cdot C_{n-1,0}(q)
C_{n,1}'(q) = -2(q{-}1) \cdot C_{n-1,1}(q) - (q{-}1)^2\cdot C_{n-1,1}'(q) - C_{n-1,0}(q) - (q{-}1)\cdot C_{n-1,0}'(q),
\end{equation*}
whence
\begin{multline*}
  c_{n,1,1} = C_{n,1}'(0) = 2\cdot C_{n-1,1}(0) - C_{n-1,1}'(0) -
  C_{n-1,0}(0) + C_{n-1,0}'(0)
\\
   = 2\cdot(-1)^{n-2}-c_{n-1,1,1} - 0 + (-1)^{n-1}.
\end{multline*}
Thus, we have the recursive formula $c_{n,1,1} = (-1)^n - c_{n-1,1,1}$. Starting from $C_{2,1}(q) = (q-1)^3$ and $c_{2,1,1} = 3$, we obtain
the closed form expression $c_{n,1,1} = (-1)^n(n{+}1)$.

Concerning $C_{n,2}(q)$, differentiate \eqref{cn2} with respect to $q$, obtaining
\begin{equation*}
  %  C_{n,2}(q) &=  q\cdot C_{n-1,1}(q) 
C_{n,2}'(q) = C_{n-1,1}(q) + q \cdot C_{n-1,1}'(q),
\end{equation*}
whence 
\begin{equation*}
  c_{n,2,1} = C_{n,2}'(0) = C_{n-1,1}(0),\quad \text{and} \quad c_{n,2,1} = c_{n-1,1,0} = (-1)^{n-2}.
\end{equation*}

Concerning $C_{n,12}(q)$, differentiate \eqref{cn12} with respect to $q$, obtaining
\begin{equation*}
  %  C_{n,12}(q)  &= (q{-}1)\cdot C_{n-1,1}(q) + C_{n-1,0}(q)
C_{n,12}'(q) = C_{n-1,1}(q) + (q{-}1)\cdot C_{n-1,1}'(q) + C_{n-1,0}'(q),
\end{equation*}
whence
\begin{multline*}
  c_{n,12,1} = C_{n,12}'(0) = C_{n-1,1}(0) - C_{n-1,1}'(0) + C_{n-1,0}'(0) 
       \\
 = c_{n-1,1,0} - c_{n-1,1,1} + c_{n-1,0,1} = (-1)^{n-2} -(-1)^{n-1}n +  (-1)^{n-1} = (-1)^nn.
\end{multline*}
At last, concerning $C_{n,21}(q)$, we have from \eqref{cn21}
\begin{multline*}
  % C_{n, 21}(q) &= -(q{-}1)\cdot C_{n-1,2}(q) + q\cdot C_{n-1,12}(q)   - (q{-}1)^2\cdot C_{n-1,21}(q)
C_{n,21}'(q) = - C_{n-1,2}(q) - (q{-}1)\cdot C_{n-1,2}'(q)
\\ 
               + C_{n-1,12}(q) + q\cdot C_{n-1,12}'(q)
\\  
               -2(q{-}1)\cdot C_{n-1,21}(q) - (q-1)^2\cdot C_{n-1,21}'(q),
\end{multline*}
so, since $C_{n-1,2}(0) = 0$ and $C_{n-1,21}(0) = 0$,
\begin{align*}
  c_{n, 21,1} &= C_{n,21}'(0) = -(-1)\cdot C_{n-1,2}'(0) + C_{n-1,12}(0) - (-1)^2\cdot C_{n,21}'(0)
\\
   &= c_{n-1,2,1} + c_{n-1,12,0} - c_{n-1,21,1} = (-1)^{n-3} +  (-1)^{n-2} - c_{n-1,21,1}
\\
   &= -c_{n-1,21,1}.
\end{align*}
Starting from $C_{2,21}(q) = q$ and $c_{2,21,1} = 1$, we deduce $c_{n,21,1}= (-1)^n$.
\end{proof}
\begin{theorem}
  \label{palindromes}
The following identities are satisfied by the polynomials $C_{n,-}(q)$.
\begin{gather}
C_{n,0}(q) = q^{2n}C_{n,0}(q^{-1}), \quad
C_{n,1}(q) = -q^{2n-1}C_{n,1}(q^{-1}), \quad
C_{n,2}(q) = -q^{2n-1}C_{n,2}(q^{-1}),   \notag
\\
C_{n,12}(q) = q^{2n-2} C_{n,12}(q^{-1}),\quad
C_{n,21}(q) = q^{2n-2}C_{n,21}(q^{-1}).  \label{palindromicidentities}
\end{gather}
In terms of the coefficients of the various polynomials, we have
\begin{gather}
  c_{n,0,i} = c_{n,0,2n-i}, \quad c_{n,1,i} = -c_{n,1,2n-1-i}, \quad c_{n,2,i} = -c_{n,2,2n-1-i},  \notag
\\
  c_{n,12,i} = c_{n,12, 2n-2-i}, \quad c_{n,21,i} = c_{n,21,2n-2-i}.  \label{coeffidentities}
\end{gather}
\end{theorem}
These identities reflect certain palindromic properties of the polynomials and permit us to compute their degrees in
corollary \ref{degrees1}.   We will say a polynomial $p(x) = a_0 + a_1\, x + \cdots + a_n\,x^n$ of degree $n$ is {\em palindromic} if
\begin{equation*}
  p(x) = x^np(x^{-1}) = x^n(a_0 + a_1\,x^{-1} + \cdots + a_n\,x^{-n}) = a_n + a_{n-1}\,x + \cdots +  a_0\, x^n.
\end{equation*}
We say a polynomial of degree $n$ is {\em skew-palindromic} if $p(x) = - x^np(x^{-1})$.
Obviously, for a palindromic or a skew-palindromic polynomial the leading and trailing coefficients are both non-vanishing.  
\begin{proof}
These are all proved by induction, using the recursive formulas from theorem \ref{heckerecursion}.  
For the first identity,
\begin{align*}
  C_{n,0}(q) &= q^2\cdot C_{n-1,21}(q) - q(q{-}1)\cdot C_{n-1,1}(q)
\\
& = q^2 \cdot q^{2n-4}C_{n-1,21}(q^{-1}) - q(q{-}1)\cdot \bigl(-q^{2n-3}C_{n-1,1}(q^{-1})\bigr)
\\
     &= q^{2n}\cdot\bigl(q^{-2}C_{n-1,21}(q^{-1})\bigr) - q^{2n}\cdot\bigl(q^{-1}(q^{-1}-1)\cdot C_{n-1,1}(q^{-1})\bigr) = q^{2n}\cdot C_{n,0}(q^{-1}).
\end{align*}
For the second identity, 
\begin{align*}
  C_{n,1}(q) &= - (q{-}1)^2\cdot C_{n-1,1}(q) - (q{-}1)\cdot C_{n-1,0}(q)
\\
            &= -(q{-}1)^2\cdot (-1) \cdot q^{2n-3}C_{n-1,1}(q^{-1}) -(q{-}1)\cdot q^{2n-2}C_{n-1,0}(q^{-1})
\\
            &= q^{2n-1}\cdot\bigl((1{-}q^{-1})^2\cdot C_{n-1,1}(q^{-1}) + (q^{-1}{-}1)\cdot C_{n-1,0}(q^{-1})\bigr)
\\ 
            &= -q^{2n-1}\cdot C_{n,1}(q^{-1}).
\end{align*}
For the third identity,
\begin{multline*}
  C_{n,2}(q) =q\cdot C_{n-1,1}(q) 
           =q\cdot(-1)\cdot q^{2n-3}\cdot C_{n-1,1}(q^{-1})
           \\
           = -q^{2n-1}\cdot\bigl( q^{-1}\cdot  C_{n-1,1}(q^{-1})
           = -q^{2n-1}\cdot C_{n,2}(q^{-1}).
\end{multline*}
For the fourth identity,
\begin{align*}
  C_{n,12}(q) &= (q{-}1)\cdot C_{n-1,1}(q) + C_{n-1,0}(q)
\\
             &=(q{-}1)\cdot (-1)\cdot q^{2n-3} \cdot C_{n-1,1}(q^{-1}) + q^{2n-2}\cdot C_{n-1,0}(q^{-1})
\\
             &=q^{2n-2}\cdot \bigl(q^{-1}{-}1)\cdot C_{n-1,1}(q^{-1}) + C_{n-1,0}(q^{-1}) \bigr)
\\
             &= q^{2n-2}\cdot C_{n,12}(q^{-1}).
\end{align*}
Finally, for the fifth identity,
\begin{align*}\begin{split}
  C_{n,21}(q) &= -(q{-}1)\cdot C_{n-1,2}(q) + q\cdot C_{n-1,12}(q) - (q{-}1)^2\cdot C_{n-1,21}(q)
\\ 
  &= -(q{-}1)\cdot (-1)\cdot q^{2n-3}\cdot C_{n-1,2}(q^{-1}) + q \cdot q^{2n-4} \cdot C_{n-1,12}(q^{-1})
\\ 
  &\hspace{18em}  - (q{-}1)^2\cdot q^{2n-4}\cdot C_{n-1,21}(q^{-1})
\\
 &= q^{2n-2}\cdot \bigl( q^{-1}{-}1)^2 \cdot C_{n-1,2}(q^{-1}) + q^{-1}\cdot C_{n-1,12}(q^{-1}) - (q^{-1}{-}1)^2\cdot C_{n-1,21}(q^{-1})\bigr)
\\
 &= q^{2n-2}\cdot C_{n,21}(q^{-1}).   \qedhere
\end{split}
\end{align*}
\end{proof}
Now we state the implications for the polyomials $C_{n,*}(q)$.
\begin{corollary}
  \label{degrees1}
The degrees of the polynomials $C_{n,*}(q)$ are as follows.
\begin{gather}
 \deg(C_{n,0}) = 2n{-}1, \quad \deg(C_{n,1}) = 2n{-}1 , \quad \deg(C_{n,2}) = 2n{-}2,  \notag
\\
\deg(C_{n,12}) = 2n{-}2, \quad \text{and} \quad \deg(C_{n,21}) = 2n{-}3.   \label{degreevalues}
\end{gather}
\end{corollary}
\begin{proof}
Consider first $C_{n,0}(q)$.   We know from proposition \ref{trailing} that $C_{n,0}(0){=}0$, so $q$ is a
factor.  That is, $P_{n,0}(q) = q^{-1}C_{n,0}(q)$ is also a polynomial, and its trailing coefficient is $c_{n,0,1} \neq 0$ by proposition
\ref{degreeone}. We observe
\begin{equation*}
  P_{n,0}(q) = q^{-1}C_{n,0}(q) = q^{-1}\cdot q^{2n}C_{n,0}(q^{-1}) = q^{2n-2}\cdot q\,C_{n,0}(q^{-1}) = q^{2n-2}\cdot P_{n,0}(q^{-1}).
\end{equation*}
Thus, $C_{n,0}(q)$ is $q$ times a palindromic polynomial of degree $2n{-}2$, which means $C_{n,0}(q)$ has degree $2n{-}1$. 
Similarly, we conclude that $C_{n,2}(q)$ and $C_{n,21}(q)$ are, respectively, $q$ times a skew-palindromic polynomial of 
degree $2n{-}3$ and $q$ times a palindromic polynomial of degree $2n{-}4$.  Thus $C_{n,2}(q)$ has degree $2n{-}2$ and
$C_{n,21}(q)$ has degree $2n{-}3$. Since $C_{n,1}(0){\neq}0$ and $C_{n,12}(0){\neq}0$, the identities stated in theorem
yield that $C_{n,1}(q)$ is a skew-palindromic polynomial of degree $2n{-}1$ and that $C_{n,12}(q)$ is a palindromic polynomial
of degree $2n{-}2$.  
\end{proof}
 Accordingly, set
\begin{align}
C_{n,0}(q) &= \sum_{i=1}^{2n-1} c_{n,0,i}q^i, &
  C_{n,1}(q) &= \sum_{i=0}^{2n-1} c_{n,1,i}q^i, & 
  C_{n,2}(q) &= \sum_{i=1}^{2n-2} c_{n,2,i}q^i,   \notag
\\
  C_{n,12}(q) &= \sum_{i=0}^{2n-2} c_{n,12,i}q^i,
& &\text{and} & 
 C_{n,21}(q) &= \sum_{i=1}^{2n-3} c_{n,21,i}q^i.  \label{Cexpansions}
\end{align}
\section{Obtaining  the Polynomial Invariants}   \label{Polys}
Following the construction given in 
\cite[p.288]{Jones_poly86}
we work over the function field $K = \bC(q,z)$
and follow their recipes to obtain expressions for the two-variable HOMFLY-PT polynomials,
the one-variable Jones polynomials $V_{W(3,n)}(t)$,
and the Alexander polynomials $\Delta_{W(3,n)}(t)$.
The expressions are subsequently refined to incorporate information obtained in section \ref{Hecke}.
From this point we evaluate the span of the Jones polynomial $V_{W(3,n)}(t)$ in proposition \ref{span}, a 
result already known to Kauffman \cite[Theorem~2.10]{States}, where we demonstrate how to use 
equations \eqref{trailingcoefftable} and \eqref{coeffidentities}.

Let $H_{N+1}$ be the Hecke algebra over $K$ corresponding
to $q$ with $N$ generators as in definition \ref{Heckealgebras}. The
starting point is the following theorem.
\begin{theorem}
  \label{traces}
For $N \geq 1$ there is a family of trace functions $\Tr \colon H_{N+1} \ra K$ compatible with the inclusions $H_N \ra H_{N+1}$ satisfying
\begin{enumerate}
 \item $\Tr(1) = 1$,
\item $\Tr$ is $K$-linear and $ \Tr(ab) = \Tr(ba)$,
\item If $a, b \in H_N$, then $\Tr (aT_Nb) = z\Tr(ab)$.
\end{enumerate}
\end{theorem}
Property 3 enables the calculation of $\Tr$ on basis elements of $H_{N+1}$ through use of the defining relations and induction. 
For $H_3$, note that
\begin{equation*}
  \Tr(T_1) = \Tr(T_2) = z, \quad \Tr(T_1T_2) = \Tr(T_2T_1) = z^2, \quad \Tr(T_1T_2T_1) = z \Tr(T_1^2) = z \bigl((q{-}1)z + q\bigr),
\end{equation*}
and we put $w=1{-}q{+}z$.  
The next step toward the polynomial invariants of the knot that is the closure of
the braid $\alpha \in B_{N+1}$ is given by the formula
\begin{equation*}
  V_{\alpha}(q,z) = \Bigl(\frac{1}{z}\Bigr)^{(N + e(\alpha))/2}
                        \cdot \Bigl( \frac{q}{w} \Bigr)^{(N-e(\alpha))/2}\cdot \Tr\bigl(\rho(\alpha)\bigr),
\end{equation*}
where $e(\alpha)$ is the exponent sum of the word $\alpha$. The expression defines an element in the quadratic extension $K(\sqrt{q/zw})$. 
For the weaving knot $W(3,n)$, viewed as the closure of $(\sigma_1\sigma_2^{-1})^n$, we have the exponent sum $e=0$, and $N=2$, and 
\begin{equation*}
 \rho\bigl( (\sigma_1\sigma_2^{-1})^n \bigr)= (T_1T_2^{-1})^n 
        = q^{-n}\bigl( C_{n,0}(q) +  C_{n,1}(q)\cdot T_1 + C_{n,2}(q) \cdot T_2 + C_{n,12}(q) \cdot T_1T_2  + C_{n,21}(q) \cdot T_2T_1\bigr) ,
\end{equation*}
thanks to proposition \ref{C121}, which says the expression for $(T_1T_2^{-1})^n$ requires only the use of the basis elements
$1$, $T_1$, $T_2$,  $T_1T_2$ and $T_2T_1$. Then we have 
\begin{multline}
  V_{(\sigma_1\sigma_2^{-1})^n}(q,z) \\  = \Bigl(\frac{1}{z}\Bigr)\cdot \Bigl( \frac{q}{w} \Bigr)\cdot q^{-n}
     \Tr \bigl(C_{n,0}(q) +  C_{n,1}(q)\cdot T_1 + C_{n,2}(q) \cdot T_2 + C_{n,12}(q) \cdot T_1T_2  
               + C_{n,21}(q) \cdot T_2T_1   \bigr)
\\
  = \Bigl(\frac{q}{zw}\Bigr)\cdot   q^{-n} \cdot \bigl( C_{n,0}(q) +   C_{n,1}(q) \cdot z + C_{n,2}(q) \cdot z + C_{n,12}(q) \cdot z^2
               + C_{n,21} (q) \cdot z^2 \bigr), \label{Heckeoutput}
\end{multline}
using the facts that $\Tr T_1 = \Tr T_2 = z$ and $\Tr T_1T_2 = \Tr T_2T_1 = z^2$. 
This expression is the starting point for our manipulations. 

Following \cite{Jones_poly86}, we point out that the universal skein invariant 
$P_{W(3,n)}(\ell, m)$, an element of the Laurent polynomial ring $\bZ[\ell, \ell^{-1}, m , m^{-1}]$,  is obtained by rewriting
$V_{(\sigma_1\sigma_2^{-1})^n}(q,z)$ in terms of $\ell = i(z/w)^{1/2}$ and $m = i(q^{-1/2} - q)$, a task
easily managed in a computer algebra system by simplifying $V_{(\sigma_1\sigma_2^{-1})^n}(q,z)$ with
respect to side relations.  Starting from 
$P_{W(3,n)}(\ell, m)$, the Jones polynomial $V_{W(3,n)}(t)$ is obtained by 
setting $\ell = it $ and $m = i (t^{1/2}{-}t^{-1/2}) $, the Alexander polynomial $\Delta_{W(3,n)}(t)$ is obtained
by setting $\ell = i$ and $m = i (t^{1/2}{-}t^{-1/2})$, and the HOMFLY-PT polynomial is obtained by
setting $\ell = ia$ and $m = iz$.  We have no specific use for the HOMFLY-PT polynomial
in this paper, so we content ourselves with a few values in section \ref{PolynomialsExtra}. 

To obtain the  Alexander polynomial  from $V_{(\sigma_1\sigma_2^{-1})^n}(q,z)$, it is useful to first rewrite
\begin{equation*}
  V_{(\sigma_1\sigma_2^{-1})^n}(q,z)
  =  q^{-n+1} \cdot \bigl( C_{n,0}(q)\cdot(zw)^{-1} +   \bigl(C_{n,1}(q){+}C_{n,2}(q)\bigr) \cdot w^{-1} 
       + \bigl(C_{n,12}(q) {+} C_{n,21} (q)\bigr) \cdot zw^{-1} \bigr).
\end{equation*}
First set $q=t$ and make the substitutions
\begin{equation*}
  z = \frac{\ell^2(t-1)}{1+\ell^2}, \quad w = 1- q + z = 1- t + z = \frac{-1(t{-}1)}{1+\ell^2}
\end{equation*}
to obtain an expression
\begin{multline*}
  t^{-n+1}\cdot \Bigl( C_{n,0}(t) \cdot \biggl(\frac{(1+\ell^2)}{\ell^2}\biggr)^2\cdot \frac{1}{(t-1)t}
   + \bigl(C_{n,1}(t) + C_{n,2}(t)\bigr)\cdot \frac{(-1)\dot( 1 + \ell^2)}{t-1} 
\\ 
          + \bigl(C_{n,12}(t)    + C_{n,21}(t)\bigr)\cdot (-1) \cdot \ell^2\Bigr)
\end{multline*}
Now make the substitution $\ell = i$ and we arrive at
\begin{equation}
  \label{eq:alexander}
  \Delta_{W(3,n)}(t) = t^{-n+1}  \bigl(C_{n, 12}(t) + C_{n, 21}(t)\bigr).
\end{equation}
Evidently a lot of information from the braid representation of $W(3,n)$ has been lost. 
To see what remains, corollary \ref{degrees1} says that the degree of $C_{n,12}(t)$ is $2n{-}2$ and
the degree of $C_{n,21}(t)$ is $2n{-}3$. It follows that the degree of $\Delta_{W(3,n)}(t)$ is 
$(2n{-}2) - n + 1= n{-}1$.  Moreover, the lowest order non-vanishing coefficients are
$c_{n, 12, 0} = (-1)^n $ and $c_{n,21,1} = (-1)^n$. By theorem \ref{palindromes} we also have
$c_{n, 12, 2n-2} = c_{n, 12, 0} = (-1)^n$ and $c_{n, 21, 2n-3} = c_{n, 21, 1}= (-1)^n$. Thus,
\begin{equation}
  \label{eq:alexanderw3n}
  \Delta_{W(3,n)}(t) =  a_0 + \sum_{s>0} a_s(t^s+t^{-s}) = (-1)^n \cdot t^{-n+1} + \cdots + (-1)^n\cdot t^{n-1}.
\end{equation}
\begin{theorem}
  \label{HFhomology}
The Seifert genus of $W(3,n)$ is $n{-}1$, and the complement of $W(3,n)$ is fibered over $S^1$.
\end{theorem}
\begin{proof}
We know the signature   of $W(3,n)$ is zero, by Proposition \ref{weavingsignature}, so 
we apply \cite[Theorem 1.3]{OSFloer} relating the coefficients of the Alexander polynomial and the signature 
of $W(3,n)$ to the ranks of the Heegard-Floer homology groups of $S^3$ associated to $W(3,n)$.  The result is
  \begin{equation*}
    \widehat{HFK}_s(S^3, W(3,n), s) =
    \begin{cases}
      \bZ^{\abs{a_s}}, \quad \text{$0 \leq s \leq n{-}1$},
 \\
      0, \quad \text{else}.
    \end{cases}
  \end{equation*}
By \cite[Theorem 1.2]{OSFloer2}, Seifert genus of $W(3,n)$ is $n{-}1$.  
Since we have explicitly
\begin{equation*}
   \widehat{HFK}_s(S^3, W(3,n), n{-}1) \iso \bZ,
\end{equation*}
 \cite[Theorem 2.5]{Manolescu} says that the complement of  $W(3,n)$ is fibered over $S^1$.
\end{proof}

Turning to the Jones polynomial, we follow a similar scheme, but the details are necessarily more complicated.  To start, the substitutions
\begin{equation*}
  q = t, \quad  z = \frac{t^2}{1+t}, \quad w = \frac{1}{1+t}
\end{equation*}
in \eqref{Heckeoutput} lead to the one-variable Jones polynomial
\begin{multline*}
  V_{W(3,n)}(t) = \frac{t(1{+}t)^2}{t^2} \cdot t^{-n} \cdot
     \Bigl( C_{n,0}(t) + (C_{n,1}(t)+C_{n,2}(t))\cdot \frac{t^2}{1{+}t} + (C_{n,12}(t) + C_{n,21}(t)) \cdot \frac{t^4}{(1{+}t)^2}\Bigr)
 \\
  = t^{-n-1}\cdot\bigl( (1{+}t)^2\cdot C_{n,0}(t) + (1{+}t)\cdot( C_{n,1}(t) + C_{n,2}(t) )\cdot t^2 + (C_{n,12}(t) + C_{n,21}(t))\cdot t^4 \bigr).
\end{multline*}
\begin{example}
  For $W(3,1)$, which is the unknot, we have
  \begin{align*}
 V_{W(3,1)}(t) 
   &=  t^{-2}\cdot\bigl( (1{+}t)^2\cdot C_{1,0}(t) + (1{+}t)\cdot( C_{1,1}(t) + C_{1,2}(t) )\cdot t^2 + (C_{1,12}(t) + C_{1,21}(t))\cdot t^4 \bigr) 
\\
  &=  t^{-2}\cdot\bigl( (1{+}t)^2\cdot 0 + (1{+}t)\cdot(-(t-1) + 0 )\cdot t^2 + (1 + 0 )\cdot t^4 \bigr)
\\
  &= t^{-2}\cdot ( (1{-}t^2) t^2 + t^4 ) = 1.
  \end{align*}
\end{example}
\begin{example} \label{jonesfig8knot}
  For $W(3,2)$, which is the figure-8 knot, we have
  \begin{align*}
    V_{W(3,2)}(t) 
&=t^{-3}\cdot \bigl( (1{+}t)^2\cdot C_{2,0}(t) + (1{+}t)\cdot( C_{2,1}(t) + C_{2,2}(t) )\cdot t^2 
                              + (C_{2,12}(t) + C_{2,21}(t))\cdot t^4 \bigr)
\\
&=t^{-3}\cdot \bigl( (1{+}t)^2\cdot t(t{-}1)^2 +(1{+}t)\cdot( (t{-}1)^3 - t(t{-}1)  ) \cdot t^2
                              +( -(t{-}1)^2+   t ) \cdot t^4 \bigr) 
\\
&= t^{-3}\cdot \bigl( t^5 - t^4 + t^3 -t^2 + t \bigr) = t^2 - t + 1 -t^{-1} + t^{-2}
  \end{align*}
\end{example}

Now we take a closer look at the formal expression 
\begin{multline*}
  V_{W(3,n)}(t) =
 \\
  = t^{-n-1}\cdot\bigl( (1{+}t)^2\cdot C_{n,0}(t) + (1{+}t)\cdot( C_{n,1}(t) + C_{n,2}(t) )\cdot t^2 + (C_{n,12}(t) + C_{n,21}(t))\cdot t^4 \bigr)
\end{multline*}
for the Jones polynomial of the weaving knot $W(3,n)$. 

% \begin{lemma} \label{Cn0coefficients}
%   In the polynomial $C_{n,0}(q)$, the constant term $c_{n,0,0} = 0$  for all   $n \geq 1$, 
%   and the degree one coefficient $c_{n,0,1}=(-1)^{n-2}$ for $n \geq 2$.
% \end{lemma}
% \begin{proof}
%   The first polynomial $C_{1,0}(q) = 0$, and setting $q=0$ in the recurrence relation \eqref{cn0} immediately yields
%   \begin{equation*}
%     c_{n,0,0} = C_{n,0}(0) = 0.
%   \end{equation*}
% Differentiate the recursion relation \eqref{cn0} with respect to $q$, obtaining
% \begin{equation*}
%   C_{n,0}'(q) = \bigl( 2q\cdot C_{n-1,21}(q)+ q^2\cdot C_{n-1,21}'(q) \bigr)  
%               - \bigl( (2q{-}1)\cdot C_{n-1,1}(q)+ q(q{-}1) \cdot C_{n-1,1}'(q)\bigr) .
% \end{equation*}
% Substituting $q=0$ yields immediately $c_{n,0,1} = C_{n,0}'(0) = C_{n-1,1}(0) = c_{n-1,1,0}$.  We have 
% \begin{equation*}
%    C_{n,1}(q) = - (q{-}1)^2\cdot C_{n-1,1}(q) - (q{-}1)\cdot C_{n-1,0}(q), 
% \end{equation*}
% simplifying relation \eqref{cn1} using proposition \ref{C121} to set $C_{n-1,121}(q)=0$. 
% Now we prove $c_{n,1,0} = (-1)^{n-1}$ for all $n \geq 1$. 
% We have $C_{1,1}(q) = -(q-1)$, so $c_{1,1,0} = 1$ as claimed.  Substituting $q=0$ and using $c_{n,0,0}=0$ we get
% \begin{equation*}
%   c_{n,1,0} = C_{n,1}(0) = -(-1)^2\cdot C_{n-1,1}(0) - (-1) \cdot C_{n-1,0}(0) = -c_{n-1,1,0} + 0 = -(-1)^{n-2}=(-1)^{n-1}
% \end{equation*}
% Thus $c_{n,0,1} = c_{n-1,1,0} = (-1)^{n-2}$, as claimed.
% \end{proof}
%Thus, we improve the expression for $C_{n,0}(q)$ slightly, obtaining $C_{n,0}(q) = \sum_{i=1}^{2n-1} c_{n,0,i}q^i$. Now we examine
Incorporating the formal expansions given in equations \eqref{Cexpansions}, we have
\begin{align}
 V_{W(3,n)}(t) 
&= t^{-n-1}\cdot \bigl( (1{+}t)^2\cdot C_{n,0}(t) + (t^2{+}t^3)\cdot( C_{n,1}(t) + C_{n,2}(t) ) + t^4 \cdot (C_{n,12}(t) + C_{n,21}(t)) \bigr) \notag
\\
\begin{split}
  &=  t^{-n-1}\cdot \Biggl( (1{+}t)^2\cdot \biggl(\sum_{i=1}^{2n-1} c_{n,0,i}t^i\biggr) 
\\
     & \hspace{0.15\linewidth} + (t^2{+}t^3)\cdot \Bigl( \sum_{i=0}^{2n-1} c_{n,1,i}t^i + \sum_{i=1}^{2n-2} c_{n,2,i}t^i \Bigr) 
\\
        & \hspace{0.30\linewidth} + t^4 \cdot \biggl(\sum_{i=0}^{2n-2} c_{n,12,i}t^i + \sum_{i=1}^{2n-3} c_{n,21,i}t^i\biggr) \Biggr). \label{Vformal1}
\end{split} 
\\
&= t^{-n-1}\cdot P(t) =  t^{-n-1}\cdot( p_0 + p_1\,t + p_2 \, t^2 + p_3\, t^3 + p_4\, t^4 + \cdots). \label{Vformal2}
\end{align}
The first piece of information about the Jones polynomial $V_{W(3,n)}(t)$ we obtain by using results about the $C_{n,-}$ is the following
fact, due to Kauffmann \cite[Theorem~2.10]{States}.
\begin{proposition} \label{span}
  The span of the Jones polynomial $V_{W(3,n)}(t)$ is $2n$, and the trailing and leading coefficients are $(-1)^n$.
\end{proposition}
\begin{proof}
We observe that $  p_0 = P(0) = 1\cdot c_{n,0,0} = 0 $ by proposition \ref{trailing}, so the lowest non-zero term in $V_{W(3,n)}(t)$ is
$t^{-n-1}\cdot p_1\,t$. Clearly, $p_1 = c_{n,0,1} = (-1)^{n}$ by proposition \ref{degreeone}. 

One identifies the top degree term in the polynomial factor of \eqref{Vformal1} as the term of degree $2n{+}2$ with coefficient
\begin{equation*}
  c_{n,1,2n-1} + c_{n,12,2n-2} = -c_{n,1,0} + c_{n,12,0} = -(-1)^{n-1} + (-1)^{n-1} = 0,
\end{equation*}
where we use the palindromic equations \eqref{coeffidentities} and the table of trailing coefficients
\eqref{trailingcoefftable} to do the computation. 

Turning to the term of degree $2n{+}1$, we find the coefficient is
\begin{align*}
  c_{n,0,2n-1} &+ c_{n,1,2n-1} + c_{n,1, 2n-2} + c_{n,2,2n-2} + c_{n,12, 2n-3} + c_{n,21,2n-3}
\\
 &= c_{n,0,1} - c_{n,1,0} - c_{n,1,1} -c_{n,2,1} + c_{n,12,1} + c_{n,21,1}
\\
 &= (-1)^n - (-1)^{n-1} - (-1)^n(n{+}1) - (-1)^n + (-1)^nn + (-1)^n 
\\
 &= (-1)^n,
\end{align*}
using first \eqref{coeffidentities} and then the tables \eqref{trailingcoefftable} and \eqref{degonetable} to complete the evaluation.
\end{proof}
Thus, we may write the Jones polynomial in the form 
\begin{equation*}
V_{W(3,n)}(t) = (-1)^n\, t^{-n} + v_{-n+1}t^{-n+1} + v_{-n+2}t^{-n+2} + \cdots + v_{n-2}t^{n-2}  + v_{n-1}t^{n-1} + (-1)^n\,t^n, 
\end{equation*}
where it is known that $v_{-n+i} = v_{n-i}$, since $W(p,q)$ for $p$ odd is amphicheiral.
The twist number of $W(3,n)$ is $\abs{v_{-n+1}}{+}\abs{v_{n-1}}$ according to \cite{DasbachLin}.
 We will now recompute the twist number for $n \geq 3$ from the information we have gathered
 about the coefficients of the Jones polynomial. 
First, observe that
\begin{equation}\label{vfirst}
v_{-n+1}t^{-n+1} = t^{-n-1}\cdot p_2 \, t^2,
\end{equation}
and $p_2\,t^2$ is computed from
\begin{equation*}
  (1+2t)\cdot(c_{n,0,1}\,t + c_{n,0,2}\,t^2) + t^2\cdot c_{n,1,0} = c_{n,0,1}\, t + (c_{n,0,2}{+}2\,c_{n,0,1}{+}c_{n,1,0})\cdot t^2,
\end{equation*}
and no other terms from the expansion \eqref{Vformal1}, because $c_{n,2,0} = 0$ by proposition \ref{trailing}.
We have 
\begin{align*}
  p_2 &= c_{n,0,2}+2\,c_{n,0,1}+c_{n,1,0} = c_{n,0,2} + 2 \cdot (-1)^n + (-1)^{n-1}
\intertext{by propositions \ref{trailing} and \ref{degreeone},}
      &= c_{n,0,2} + (-1)^n.
\end{align*}
We identify $c_{n,0,2}$ by reducing the recursive description \eqref{cn0} mod $q^3$, obtaining
\begin{align*}
  c_{n,0,0} + c_{n,0,1} \, q + c_{n,0,2}\, q^2 &\equiv q^2\cdot (c_{n-1, 21,0}) + (- q^2 + q)\cdot(c_{n-1,1,0} + c_{n-1,1,1}\, q)
\\
&\equiv  c_{n-1,1,0}\,q + (c_{n-1,21,0} +c_{n-1,1,1}{-}c_{n-1,1,0})\cdot q^2 \mod q^3
\end{align*}
After extracting the coefficient of $q^2$, 
\begin{align}
  c_{n,0,2} &= c_{n-1,21,0} +c_{n-1,1,1}{-}c_{n-1,1,0} \notag
\\
  &= 0 + (-1)^{n-1}((n{-}1)+1) - (-1)^{n-2} = (-1)^{n-1}(n+1), \quad \text{for $n \geq 3$,} \label{cn0deg2}
%c_{n-1,0,2} = (-1)^{n-2}n
\end{align}
by propositions \ref{trailing} and \ref{degreeone}. Since we have used the formula for $c_{n-1,1,1}$ in \eqref{degonetable},
we must assume $n{-}1 \geq 2$. 
Referring to \eqref{vfirst},
\begin{equation*}
v_{-n+1} =  p_2= c_{n,0,2}+(-1)^n = (-1)^{n-1}n + (-1)^{n-1} + (-1)^n = (-1)^{n-1}n.
\end{equation*}
Thus, we have reproved the following formula given in theorem 5.1 of \cite{DasbachLin}.
\begin{proposition} \label{twistnumber}
For $n \geq 3$, the twist number of $W(3,n)$ is $\abs{v_{-n+1}}{+}\abs{v_{n-1}} {=} n {+} n {=} 2n$. \qed
\end{proposition}
\section{Higher Twist Numbers and Volume} \label{TwistNumbersVolume}
In \cite{DasbachLin}, Dasbach and Lin define higher twist numbers of a knot in terms of the Jones polynomial, 
with the idea that these invariants also correlate with the hyperbolic volume of the knot complement.
If 
\begin{equation*}
  V_K(t) = \lambda_{-m}t^{-m} + \lambda_{-m+1}t^{-m+1} + \cdots + \lambda_{n-1}t^{n-1} + \lambda_nt^n,
\end{equation*}
then the $j$th twist number of $K$ is $T_j(K) = \abs{\lambda_{-m+j}} + \abs{\lambda_{n-j}}$. 
Note that twist numbers $T_j(K)$ are only defined for $j$ within the span of the Jones polynomial. In the
case of weaving knots $W(3,n)$, the relevant twist numbers are defined for $1 \leq j \leq n{-}1$.
In proposition \ref{twistnumber} we have recomputed the first twist number of $W(3,n)$ using our results from section \ref{Hecke}.
In theorems \ref{twistnumber2} and \ref{twistnumber3} we extend the technique to compute the second and third twist numbers.

In the appendix to \cite{DasbachLin} one finds a scatter plot generated from a table of alternating knots of 14
crossings by plotting along a horizontal axis the higher twist numbers of the knots and along the vertical axis
the volume of the complement.  The authors also construct similar plots starting from a table of non-alternating
knots of 14 crossings.  In both cases, there appears to be some correlation between these combinatorial invariants and the 
geometric invariant.  We are going to explore how well higher twist numbers and volume correlate as the number of crossings
increases.

We supplement our rigorous calculations of $T_2\bigl(W(3,n)\bigr)$ and $T_3\bigl(W(3,n)\bigr)$ 
with some conjectural calculations in the following table of higher twist numbers. 
To obtain these results, we use {\em Mathematica} or {\em Maple}
to extract the coefficients $\lambda_{-n+k}$ of $t^{-n+k}$ in $V_{W(3,n)}(t)$ for $k = 4$, $5$, $6$, and $7$
associated to weaving knots $W(3,n)$ starting near $n = 2k$.  
We conjecture that the $k$th twist number $T_k\bigl(W(3,n)\bigr)$ is a polynomial in $n$ of degree $k$.
Taking iterated differences of the coefficient sequences, we find they are
consistent with the conjecture as long as $n$ is sufficiently large.   The following formulas for the twist numbers 
$T_k\bigl(W(3,n)\bigr) = 2 \abs{\lambda_{-n+k}}$ for $k = 4$, $5$, $6$, and $7$ were produced by fitting polynomials 
to sufficiently large selections of coefficients $\lambda_{-n+k}$ and comparing polynomial
values with computed coefficients for different values of $n$.
 \begin{table}[h!]
   \caption{Higher twist numbers for $W(3,n)$}
    \label{highertwists}
    \centering  \renewcommand{\arraystretch}{1.25}
  \begin{tabular}[h]{|l|l|}\hline
  $k$  & $T_k\bigl(W(3,n)\bigr)$
\\  \hline
    2 & $-n+n^2$ 
\\  \hline
    $3$ & $ n ( n{-}1 ) ( n{-}2 )/3 +2\,n$  
\\  \hline
    $4$ & $-(9/2)\,n+(35/12)\,{n}^{2} -(1/2)\,{n^{3}}+(1/12)\,{n^{4}}$  
\\ \hline
    $5$ & $ (42/5) \,n - (35/6)\,{n}^{2} + (19/12)\,{n}^{3} - (1/6)\,n^{4} + (1/60)\,{n}^{5}$ 
\\  \hline
    $6$ & $- (52/3) \,n + (2237/180) \,n^2 -  (29/8) \,n^3  +  (41/72) \,n^4 - (1/24)\,n^{5} +  (1/360) \, n^{6} $  
\\ \hline
    $7$ & $(254/7) \,n  -  (413/15) \,n^2 +  (1541/180)\,n^3 $ 
\\
        & $\qquad \qquad \qquad  -  (35/24)\,n^4 +  (11/72) \,n^5 - (1/120)\, n^6 + (1/2520)\, n^7$ \\ \hline
  \end{tabular}
\end{table}
Figures \ref{fig:secondtwistvsvolume}, \ref{fig:thirdtwistvsvolume}, and \ref{fig:fourthtwistvsvolume} 
plot horizontally values of the twist numbers $T_2$, $T_3$, and the conjectured $T_4$ and vertically
values of the volume of the link complement.  We used the program SnapPy \cite{SnapPy} to compute estimates of
the volume. 

In addition, 
one can ask how efficient are the bounds given in \eqref{CKPbounds} for the volume  of weaving knots $W(p,q)$.
For weaving knots $W(3,n)$ the bounds simplify to 
\begin{equation*}
  v_{{\rm oct}}\,n\,\biggl(1 - \frac{(2\pi)^2}{n^2}\biggr)^{3/2} \leq {\rm vol}(W(3,n))  <  4\,v_{{\rm tet}}\cdot n.
\end{equation*}
If we consider the volume relative to the crossing number ${\rm vol}(W(3,n))/2n$, then we have the chain
\begin{equation}  \label{CKPboundspecialrelative}
  \frac{v_{{\rm oct}}}{2}\biggl(1 - \frac{(2\pi)^2}{n^2}\biggr)^{3/2} \leq \frac{{\rm vol}(W(3,n))}{2n}  <  2\,v_{{\rm tet}}
\end{equation}
For a fixed value of $n$ there is a gap between the upper and lower bounds. 
We can ask whether or not better bounds on the relative volume of weaving knots $W(3,n)$ can be teased out of the
higher twist numbers of these knots.  
To obtain some information on the question we appeal to algorithms in the program SnapPy to generate estimates of the
volume of these knots. 

We perform the following manipulations on the formula for $T_k\bigl(W(3,n)\bigr)$.  First take the $k$th root 
of the expression and then divide by the crossing number $2n$ to obtain an expression whose limit as
$n$ tends to infinity is finite.  Then multiply by a normalization constant $C_k$ so that 
\begin{equation*}
  \lim_{n \ra \infty}C_k \frac{\sqrt[k]{T_k\bigl(W(3,n)\bigr)}}{2n} = 2 \, v_{{\rm tet}}.
\end{equation*}
In figure \ref{fig:comparingbounds}
we show the upper bound from equation \eqref{CKPboundspecialrelative} as a horizontal line at the top of the
plot and the lower bound as the lowest curve.   Values ${\rm vol}\bigl(W(3,n)\bigr)/2n$ according to SnapPy
are plotted as points.  
We also plot $C_k \cdot \sqrt[k]{T_k\bigl(W(3,n)\bigr)}/2n$ for $k=2$, $3$, and $4$.  We see that all three of these
curves provide better lower bounds on the relative volume than the lower bound given in \eqref{CKPboundspecialrelative}.
Indeed, for $n$ sufficiently large the lower bound from $T_2$ is better than the bound from $T_3$, which is, in turn,
better than the bound from $T_4$. 
\begin{figure}[h!]
  \centering
  \includegraphics[width=0.75\linewidth]{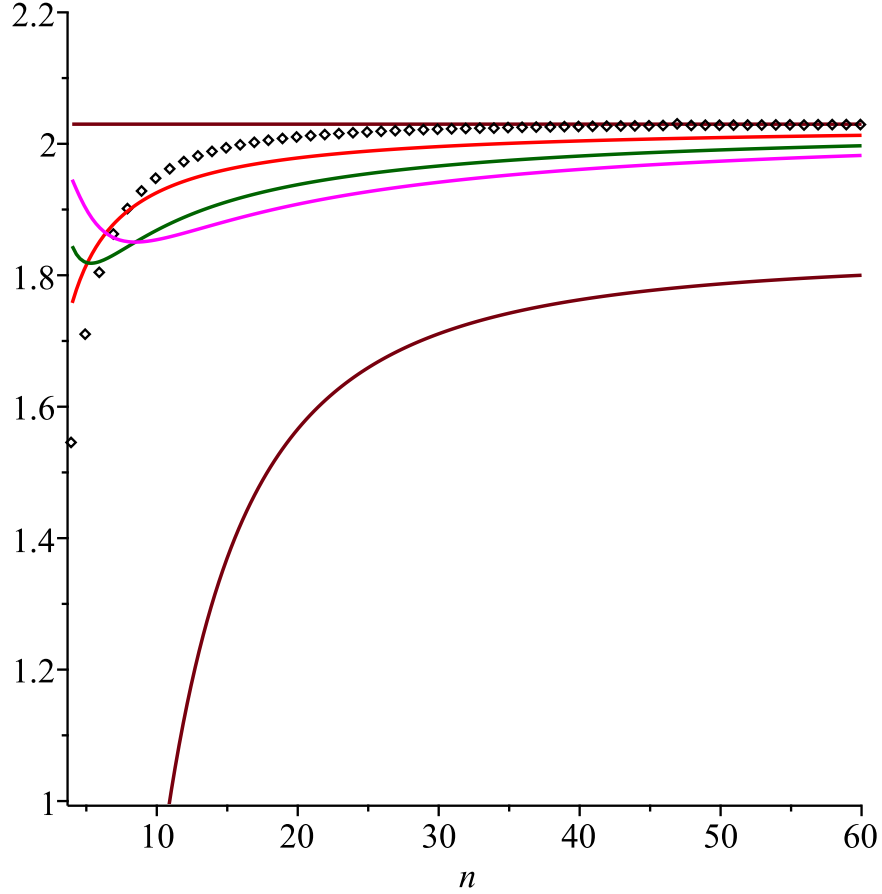}
  \caption{Comparing bounds}
  \label{fig:comparingbounds}
\end{figure}
\begin{theorem}
  \label{twistnumber2}
For $n \geq 5$, the second twist number of $W(3,n)$ is $n(n-1)$.
\end{theorem}
\begin{proof}
  Comparing \eqref{Vformal1} with \eqref{Vformal2}, and noting that $C_{n,12}(t)$ and $C_{n,21}(t)$ start in degrees $0$ and $1$, respectively,
 we want to compute the term $p_3\,t^3$ from the truncated expansion
  \begin{multline*}
    p_1\, t + p_2\,t^2 + p_3\,t^3 
\\
= (1+t)^2(c_{n,0,0} + c_{n,0,1}t + c_{n,0,2}\,t^2 + c_{n,0,3}\,t^3) 
              + (t^2+t^3)\bigl((c_{n,1,0} + c_{n,1,1}\,t) + (c_{n,2,0} + c_{n,2,1}\,t)\bigr).
  \end{multline*}
Extracting the coefficient of $t^3$ and substituting from \eqref{trailingcoefftable} and \eqref{degonetable} as well as 
equation \eqref{cn0deg2} yields
\begin{multline}   \label{p3}
p_3 =   (c_{n,0,3} + 2 c_{n,0,2} + c_{n,0,1}) + (c_{n,1,0} + c_{n,1,1}) + (c_{n,2,0} + c_{n,2,1})
\\
  =( c_{n,0,3} + 2(-1)^{n-1}(n+1)+ (-1)^n) + \bigl( (-1)^{n-1} + (-1)^n(n+1)  \bigr) + \bigl(0+ (-1)^{n} \bigr)
\\
  = c_{n,0,3} + (-1)^{n-1}n. \hspace{12em}
\end{multline}
We reduce  the recursive formula \eqref{cn0} modulo $t^4$ to compute $c_{n,0,3}$.
\begin{align*}
  C_{n,0}(t) &= t^2\cdot C_{n-1,21}(t) - t(t-1)\cdot C_{n-1,1}(t)
\\ 
   & \equiv t^2(c_{n-1,21,0} + c_{n-1,21,1}\, t) - t^2(c_{n-1,1,0} + c_{n-1,1,1}\,t) + t(c_{n-1,1,0} + c_{n-1,1,1}\,t + c_{n-1,1,2}\,t^2)
\\
   &\quad \mod t^4,
\end{align*}
 so after extracting the coefficient of $t^3$, we have
\begin{multline}  \label{cn0deg3}
  c_{n,0,3} = c_{n-1,21,1} - c_{n-1,1,1} + c_{n-1,1,2}
\\
    = (-1)^{n-1}-(-1)^{n-1}n + c_{n-1,1,2} = (-1)^{n}(n{-}1) + c_{n-1,1,2}.
%c_{n,0,3} = (-1)^{n}(n{-}1) + c_{n-1,1,2} =(-1)^{n}(n{-}1) +  (-1)^{n-2}\bigl( \frac{n(n{-}1)}{2}+1\bigr)
%c_{n-1,0,3} =  (-1)^{n{-}1}(n{-}2) + c_{n-2,1,2} = (-1)^{n{-}1}(n{-}2) +  (-1)^{n-3}\bigl( \frac{(n{-}1)(n{-}2)}{2}+1\bigr)
\end{multline}
Thus, we need a formula for $c_{n-1,1,2}$. For this return to the recursive formula \eqref{cn1} and differentiate twice, obtaining
\begin{equation*}
  %  C_{n,1}(q) &= - (q{-}1)^2\cdot C_{n-1,1}(q) - (q{-}1)\cdot C_{n-1,0}(q)
% diff once:    C_{n,1}'(q) &= -2(q{-}1)\cdot C_{n-1,1}(q)-(q-1)^2\cdotC_{n-1,1}'(q) - C_{n-1,0}- (q{-}1)\cdot C_{n-1,0}'(q)
C_{n,1}^{(2)}(q) = -2\cdot C_{n-1,1}(q) - 4(q-1)\cdot C_{n-1,1}'(q) - (q-1)^2\cdot C_{n-1,1}^{(2)}(q) - 2 C_{n-1,0}'(q) -(q-1)\cdot C_{n-1,0}^{(2)}(q).
\end{equation*}
Substituting $q=0$, we get
\begin{align*}
  2 c_{n,1,2} &= -2 c_{n-1,1,0} + 4 c_{n-1,1,1} - 2 c_{n-1,1,2} - 2 c_{n-1,0,1} +2 c_{n-1,0,2},
\\
 &=-2\cdot (-1)^{n-2} + 4 \cdot (-1)^{n-1}n - 2 c_{n-1,1,2} - 2(-1)^{n-1} + 2 (-1)^{n-2}\,n,
\end{align*}
applying propositions \ref{trailing} and \ref{degreeone} and formula \eqref{cn0deg2}.  Note that the use of 
\eqref{cn0deg2} requires $n{-}1 \geq 3$, so we have to have $n\geq 4$. 
Rewriting this expression, we get
\begin{equation}
  c_{n,1,2} + c_{n-1,1,2} = (-1)^{n-1}n, \quad \text{for $n \geq 4$.}  \label{cn1deg2recursion}
\end{equation}
Now we create a closed form expression for $c_{n,1,2}$ by forming a telescoping sum.
\begin{align*}
  c_{n,1,2} + c_{n-1,1,2} &= (-1)^{n-1}n
\\
 - c_{n-1,1,2} - c_{n-2,1,2} &= (-1)^{n-1}(n{-}1)
\\   &\cdots
\\
 (-1)^k c_{n-k,1,2} + (-1)^k c_{n-k-1,1,2} &= (-1)^k(-1)^{n-k-1}(n{-}k)
\\
     &\cdots
\\
(-1)^{n-5} c_{5,1,2} + (-1)^{n-5}c_{4,1,2} &= (-1)^{n-1}5
\\
(-1)^{n-4} c_{4,1,2} + (-1)^{n-4}c_{3,1,2} &= (-1)^{n-1}4.
\end{align*}
Adding these equations yields
\begin{multline*}
  c_{n,1,2} + (-1)^{n-4}c_{3,1,2} = (-1)^{n-1}\sum_{k=4}^n k 
\\
= (-1)^{n-1}\biggl( \frac{n(n+1)}{2} - \frac{3(3+1)}{2} \biggr)
  = (-1)^{n-1}\bigl( n(n+1)/2 - 6\bigr) .
\end{multline*}
From section \ref{PolynomialsExtra},  
$C_{3,1}(q) = 1 - 4\,q + 7\,q^2 - 7\,q^3 + 4 \, q^4 - q^5$, so $c_{3,1,2} = 7$ and  we obtain
\begin{equation}
  \label{cn1deg2}
  c_{n,1,2} = (-1)^{n-1}\Bigl( \frac{n(n+1)}{2}+1\Bigr), \quad \text{for $n \geq 4$.}
%c_{n-1,1,2} = (-1)^{n-2}\bigl( \frac{n(n{-}1)}{2}+1\bigr)
\end{equation}
Substituting the results of \eqref{cn0deg3} and \eqref{cn1deg2} into \eqref{p3}, we obtain
\begin{align}
  \label{p3final}
v_{-n+2} =  p_3  &= (-1)^{n-1}n + c_{n,0,3}  \notag
\\
 &= (-1)^{n-1}n  + \bigl( (-1)^{n}(n{-}1) + c_{n-1,1,2}\bigr)  \notag
\\
 &=  \bigl( (-1)^{n-1}n  + (-1)^{n}(n{-}1)\bigr) + (-1)^{n-2}(n-1)n/2 +(-1)^{n-2}  \notag
\\
 &= (-1)^n(n-1)n/2.
\end{align}
for $n{-}1 \geq 4$, or $n\geq 5$. 
We conclude the second twist number for $W(3,n)$ is $\abs{v_{-n+2}} + \abs{v_{n-2}}= n(n-1)$.
\end{proof}
\begin{theorem} \label{twistnumber3}
  For $n \geq 5$ the coefficient of $t^{-n+3}$ in the Jones polynomial $V_{W(3,n)}(t)$ is
\begin{equation*}
v_{-n+3} = (-1)^{n-1}\bigl(n(n-1)(n-2)/6 + n\bigr),
\end{equation*}
so third twist number for $W(3,n)$ is $\abs{v_{-n+3}} + \abs{v_{n-3}}= n(n-1)(n-2)/3 + 2n$.
\end{theorem}
\begin{proof}
  The essential point is to compute the coefficient $p_4$ in the expansion \eqref{Vformal2}.  Starting from the 
truncated polynomial expression
\begin{multline*}
  (1+2t+t^2)(c_{n,0,1}t+c_{n,0,2}t^2+ c_{n,0,3}t^3 + c_{n,0,4}t^4) 
\\
 + (t^2 + t^3)\bigl[(c_{n,1,0} + c_{n,1,1}t + c_{n,1,2}t^2) + (c_{n,2,0} + c_{n,2,1}t + c_{n,2,2}t^2)\bigr]
\\
 + t^4(c_{n,12,0}+c_{n,21,0}),
\end{multline*}
we extract the coefficient of $t^4$, obtaining 
\begin{align}
  p_4 &= (c_{n,0,4}+2\,c_{n,0,3} + c_{n,0,2}) + \bigl[(c_{n,1,1}+c_{n,1,2}) + (c_{n,2,1}+c_{n,2,2})\bigr] + (c_{n,12,0}+c_{n,21,0})\notag
\\
     &= (c_{n,0,4}+2\,c_{n,0,3} + c_{n,0,2}) + \bigl[((-1)^n(n+1)+c_{n,1,2}) + ((-1)^n+c_{n,2,2})\bigr] + (-1)^{n-1} \notag
\intertext{by \eqref{trailingcoefftable} and \eqref{degonetable},}
     &= (c_{n,0,4}+ 2\,c_{n,0,3} +(-1)^{n-1}(n{+}1)) + \bigl[((-1)^n(n+1)+c_{n,1,2}) + c_{n,2,2})\bigr]  \notag
\intertext{evaluating $c_{n,0,2} = (-1)^{n-1}(n{+}1)$ for $n \geq 3$ by \eqref{cn0deg2}, }
     &= c_{n,0,4} + 2\, c_{n,0,3} + c_{n,1,2} + c_{n,2,2}  \notag
\\
     &= c_{n,0,4} +2 \bigl( (-1)^{n}(n-1) + c_{n-1,1,2} \bigr) + c_{n,1,2} + c_{n-1,1,1} \notag
\intertext{substituting for $c_{n,0,3}$ from \eqref{cn0deg3} and using \eqref{cn2} which implies $c_{n,2,2}=c_{n-1,1,1}$,   } 
     &= c_{n,0,4} + 2(-1)^{n}(n-1) + c_{n-1,1,2} + (-1)^{n-1}n + (-1)^{n-1}n  \notag
\intertext{since $ c_{n,1,2} + c_{n-1,1,2} = (-1)^{n-1}n$ by \eqref{cn1deg2recursion} and $c_{n-1,1,1} = (-1)^{n-1}n$ by \eqref{degonetable}, }
     &= c_{n,0,4} +c_{n-1,1,2} + 2(-1)^{n-1}  \label{p4v1}
\end{align}
Compute $c_{n,0,4}$ from the recursion formula \eqref{cn0} reduced modulo $q^5$, which yields
\begin{multline*}
  c_{n,0,0} + c_{n,0,1}\,q + c_{n,0,2}\,q^2 + c_{n,0,3}\,q^3 + c_{n,0,4}\,q^4 
\\
  = q^2\bigl( c_{n-1,21,0} + c_{n-1,21,1}\, q + c_{n-1,21,2}\, q^2 \bigr) 
    - q^2\bigl( c_{n-1,1,0} + c_{n-1,1,1}\, q + c_{n-1,1,2}\,q^2 \bigr)
\\
    + q\bigl( c_{n-1,1,0} + c_{n-1,1,1}\, q + c_{n-1,1,2}\,q^2  + c_{n-1,1,3}\,q^3\bigr)
\end{multline*}
Extracting the coefficients of $q^4$ gives
\begin{equation}
  \label{cn04step1}
  c_{n,0,4} = c_{n-1,21,2} - c_{n-1,1,2} + c_{n-1,1,3},
\end{equation}
so we have 
\begin{align}
   p_4  &=   c_{n,0,4} +c_{n-1,1,2} + 2(-1)^{n-1} = (c_{n-1,21,2} - c_{n-1,1,2} + c_{n-1,1,3})+c_{n-1,1,2} + 2(-1)^{n-1} \notag
\\
        &=   c_{n-1,21,2} +c_{n-1,1,3} + 2(-1)^{n-1} \label{p4v2}
\end{align}
We now deal with $c_{n-1,21,2}$
by reducing the recurrence relation \eqref{cn21} mod $q^3$.
We get
\begin{multline*}
  c_{n,21,0} + c_{n,21,1}\,q + c_{n,21,2}\,q^2 
\\
    \equiv (-q{+}1)\bigl(c_{n-1,2,0} + c_{n-1,2,1}\,q + c_{n-1,2,2}\,q^2\bigr) + q \bigl(c_{n-1,12,0}+ c_{n-1,12,1}\, q \bigr)
\\
    -(q{-}1)^2\bigl( c_{n-1,21,0} + c_{n-1,21,1} \, q + c_{n-1,21,2}\, q^2 \bigr) \mod q^3.
\end{multline*}
Extracting the coefficient of $q^2$ gives
\begin{equation*}
  c_{n,21,2} = c_{n-1,2,2} - c_{n-1,2,1} + c_{n-1,12,1} 
   + (-c_{n-1,21,2} + 2\,c_{n-1,21,1} - c_{n-1,21,0})
\end{equation*}
or, since $C_{n-1,2}(q) = q\cdot C_{n-2,1}(q)$ by \eqref{cn2}, we have $c_{n-1,2,2}= c_{n-1,1,1}$, so 
\begin{align}
  c_{n,21,2} &= c_{n-2,1,1}- c_{n-1,2,1} + c_{n-1,12,1} + (-c_{n-1,21,2} + 2\,c_{n-1,21,1} - c_{n-1,21,0}) \notag
\\
           &= (-1)^{n-2}(n-1) - (-1)^{n-1} + (-1)^{n-1}(n-1) - c_{n-1,21,2} + 2(-1)^{n-1} - 0, \notag
\intertext{substituting from \eqref{degonetable} and \eqref{trailingcoefftable},}
           &=  - c_{n-1,21,2} + (-1)^{n-1}.
\end{align}
We compute an alternating sum of another sequence of equalities
\begin{align*}
  c_{n,21,2} + c_{n-1,21,2} &= (-1)^{n-1}
\\
(-1)\bigl(  c_{n,21,2} + c_{n-1,21,2}\bigr) &= (-1)(-1)^{n-2}
\\
  &\cdots
\\
(-1)^{k-1}\bigl(c_{n-k+1,21,2} + c_{n-k,21,2}\bigr) &= (-1)^{k-1}(-1)^{n-k}
\\
(-1)^k\bigl(c_{n-k,21,2} + c_{n-k-1,21,2}\bigr) &= (-1)^k(-1)^{n-k-1}
\end{align*}
Adding the equations we get
\begin{equation*}
  c_{n,21,2} + (-1)^kc_{n-k-1,21,2} = (-1)^{n-1}(k{+}1),
\; \text{or} \;
  c_{n,21,2} + (-1)^{n-j-1}c_{j,21,2} = (-1)^{n-1}(n{-}j),
\end{equation*}
if we write $k=n-j-1$, so that $j=n-k-1$. 
Referring to section \ref{PolynomialsExtra}, the first $j$ for which $c_{j,21,2} \neq 0$ is $j=3$, 
and $C_{3,21}(q) = -q+ 2\,q^2 - q^3$, so $c_{3,21,2} = 2$.  Substituting and rearranging,
\begin{equation}
  \label{cn21deg2}
  c_{n,21,2} = (-1)^{n-1}(n{-}3) - (-1)^nc_{3,21,2} = (-1)^{n-1}\bigl((n-3)+2 \bigr) = (-1)^{n-1}(n-1),
%c_{n-1,21,2} = (-1)^{n-2}(n-2)= (-1)^n(n-2),
\end{equation}
and this holds for $n \geq 3$.  Substituting into \eqref{p4v2}, we get
\begin{equation}
  \label{p4v3}
  p_4 = c_{n-1,1,3} + (-1)^{n-2}(n-2) + 2(-1)^{n-1} = c_{n-1,1,3} + (-1)^nn + 4(-1)^{n-1}.
\end{equation}
The most straightforward approach to computing $c_{n-1,1,3}$ is through the recursion relation \eqref{cn1}.
Reducing the relation mod $q^4$ gives
\begin{multline*}
  c_{n,1,0} + c_{n,1,1}\, q + c_{n,1,2}\,q^2 + c_{n,1,3}\, q^3
\\
  (-1+ 2q - q^2)(c_{n-1,1,0} + c_{n-1,1,1}\,q + c_{n-1,1,2}q^2 + c_{n-1,1,3}\,q^3)
\\ 
 +   (1-q) (c_{n-1,0,0} + c_{n-1,0,1}\,q + c_{n-1,0,2}\,q^2 + c_{n-1,0,3}\,q^3)  \mod q^4,
\end{multline*}
and extracting the coefficient of $q^3$ gives
\begin{align*}
  c_{n,1,3} &= -c_{n-1,1,3} + 2\, c_{n-1,1,2} - c_{n-1,1,1} + c_{n-1,0,3} - c_{n-1,0,2}
\\&=-c_{n-1,1,3} + 2\, c_{n-1,1,2} + c_{n-1,0,3},
\end{align*}
since $c_{n-1,1,1} + c_{n-1,0,2} = (-1)^{n-1}n + (-1)^{n-2}n = 0$ for $n \geq 4$ by
\eqref{degonetable} and \eqref{cn0deg2}. 
Making the substitution for $c_{n-1,0,3}$ from \eqref{cn0deg3}, 
\begin{align}  \begin{split}
  c_{n,1,3} + c_{n-1,1,3} &= 2\,(-1)^{n-2}\Bigl(\frac{n(n-1)}{2}+1\Bigr)  + c_{n-1,0,3}  
\\  
   &= (-1)^{n-2}\bigl( n(n-1)+2\bigr) + (-1)^{n-1}(n-2) + c_{n-2,1,2}                 
\\
   &= (-1)^{n-2}\bigl( n(n-1)+2\bigr) + (-1)^{n-1}(n{-}2) + (-1)^{n-3}\bigl((n{-}1)(n{-}2)/2 + 1\bigr)
\\
   &= (-1)^n\frac{n^2}{2} + (-1)^{n-1}\frac{n}{2} + (-1)^n2
% %c_{n-1,1,1} = (-1)^{n-1}n  c_{n-1,0,2} = (-1)^{n-2}n  
\end{split}
\end{align}
Now generate a telescoping sum from the chain of equalities
\begin{align*}
  c_{n,1,3} + c_{n-1,1,3} &= (-1)^n\biggl[\frac{n^2}{2} - \frac{n}{2} +2\biggr]
\\
(-1)\bigl[ c_{n-1,1,3} + c_{n-2,1,3}\bigr] &= (-1)(-1)^{n-1}\biggl[\frac{(n{-}1)^2}{2} -\frac{n{-}1}{2} + 2\biggr]
\\
 &\cdots
\\
(-1)^k\bigl[c_{n-k,1,3} + c_{n-k-1,1,3}\bigr] &=(-1)^k (-1)^{n-k}\biggl[\frac{(n{-}k)^2}{2} - \frac{(n{-}k)}{2} + 2 \biggr]
\end{align*}
Adding these equalities gives
\begin{equation*}
  c_{n,1,3} + (-1)^kc_{n-k-1,1,3} = (-1)^n\biggl[\frac{1}{2}\sum_{j=n-k}^n j^2 - \frac{1}{2} \sum_{j=n-k}^n j + (k{+}1)2\biggr]
\end{equation*}
and, if we write $n-k-1 = \ell$, so that $n-k = \ell + 1$ and $k+1 = n - \ell$, 
\begin{align}  
\begin{split}
  c_{n,1,3} + (-1)^{n-\ell -1} c_{\ell,1,3}
    &= (-1)^n\biggl[\frac{1}{2}\sum_{j=\ell+1}^n j^2 - \frac{1}{2} \sum_{j=\ell+1}^n j + (n{-}\ell)2\biggr]
    \\
    &=  (-1)^n\biggl[\frac{1}{2}\frac{n(n{+}1)(2n{+}1)}{6} - \frac{1}{2}\frac{\ell(\ell{+}1)(2\ell{+}1)}{6}
\\  &\hspace{6em}
         - \frac{1}{2} \frac{n(n{+}1)}{2} + \frac{1}{2}\frac{\ell(\ell{+}1)}{2} +  (n{-}\ell)2 \biggr]
\end{split}
\end{align}
Taking $\ell=4$, so that, from section \ref{PolynomialsExtra},
$  C_{4,1}(q) = -1 + 5\,q - 11 \, q^2 + 16\, q^3 - \cdots$ and $c_{4,1,3} = 16$, 
we evaluate
\begin{align*}
\begin{split}
  c_{n,1,3} - (-1)^n(16) &=
  (-1)^n\biggl[\frac{1}{2} \frac{2n^3}{6} + \frac{1}{2}\frac{3n^2}{6} +  \frac{1}{2} \frac{n}{6} - \frac{1}{2}\frac{4\cdot 5 \cdot 9}{6}
\\
    &\hspace{8em}  - \frac{1}{2}\frac{n^2}{2} - \frac{1}{2}\frac{n}{2} + \frac{1}{2}\frac{4\cdot 5}{2} + 2(n-4)\biggr]
\\
&= (-1)^n \biggl[\frac{n^3}{6} - \frac{n}{6} - 10  + 2n -8 \biggr]
\\
&= (-1)^n \biggl[ \frac{n(n-1)(n+1)}{6} + 2n - 18 \biggr]
\end{split}
\end{align*}
Therefore, 
\begin{equation}
  \label{cn1deg3}
  c_{n,1,3} = (-1)^n\biggl[ \frac{n(n-1)(n+1)}{6} + 2n - 18 + 16 \biggr] = (-1)^n\biggl[ \frac{n(n-1)(n+1)}{6} + 2n - 2\biggr],
\end{equation}
a formula valid for $n \geq 4$, but failing for $n = 3$, so 
\begin{align}
   \lambda_{-n+ 3} = p_4 &= c_{n-1,1,3} + (-1)^nn + 4(-1)^{n-1} \notag
\\
&= (-1)^{n-1}\biggl[ \frac{n(n-1)(n-2)}{6} + 2(n-1) - 2\biggr] - (-1)^{n-1}n + 4(-1)^{n-1} \notag
\\
&= (-1)^{n-1}\bigl( n(n-1)(n-2)/6 + n \bigr),   \label{p4v4}
\end{align}
which is consequently valid for $n \geq 5$. This ends the proof.
\end{proof}
\begin{figure}[h!]
  \centering
  \includegraphics[width = 0.3\linewidth]{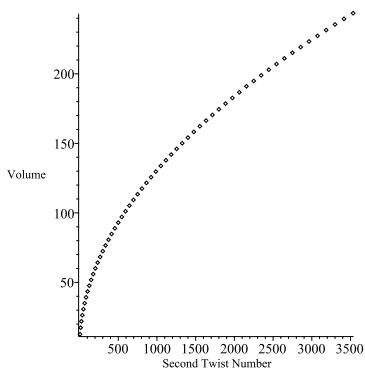}
  \caption{Second twist number versus volume}
  \label{fig:secondtwistvsvolume}
\end{figure}
\begin{figure}[h!]
  \centering
  \includegraphics[width = 0.3\linewidth]{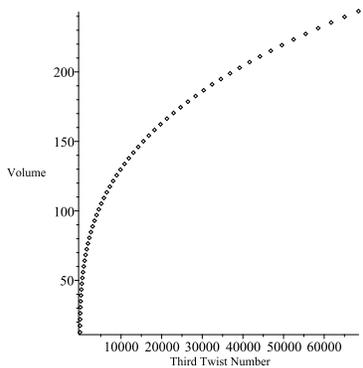}
  \caption{Third twist number versus volume}
  \label{fig:thirdtwistvsvolume}
\end{figure}
\begin{figure}[h!]
  \centering
  \includegraphics[width = 0.3\linewidth]{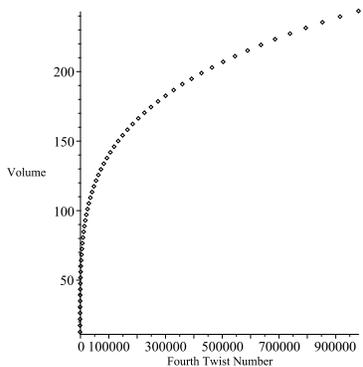}
  \caption{Fourth twist number versus volume}
  \label{fig:fourthtwistvsvolume}
\end{figure}
\section{From the Jones Polynomial to Khovanov homology}  \label{Jones-to-Khovanov}
In this section we amplify Theorem \ref{locateKH}, at least the first part of it. \addtocounter{section}{-3} \addtocounter{theorem}{4}
\begin{theorem}
For a weaving knot $W(2k{+}1,n)$ the non-vanishing Khovanov homology ${\mathcal H}^{i,j}\bigl( W(2k{+}1, n) \bigr)$ lies on the 
lines
\begin{equation*}
  j = 2i \pm 1.
\end{equation*}
For a weaving knot $W(2k, n)$ the non-vanishing Khovanov homology ${\mathcal H}^{i,j}\bigl( W(2k, n) \bigr)$ lies on the lines
\begin{equation*}
  j = 2i + n -1 \pm 1
\end{equation*}
 \end{theorem}
We have the following definition of the bi-graded Euler characteristic
associated to Khovanov homology. 
 \begin{equation*}
   Kh(L)(t,Q)  \stackrel{{\rm def}}{=} \sum t^iQ^j \dim {\mathcal H}^{i,j}(L)
 \end{equation*}
\addtocounter{section}{3} \addtocounter{theorem}{-5}
 \begin{theorem}[Theorem 1.1, \cite{Lee_Endo04}] \label{Lee1}
For an oriented link $L$, the graded Euler characteristic
\begin{equation*}
  \sum_{i,j \in \bZ} (-1)^iQ^j  \dim {\mathcal H}^{i,j}(L) 
\end{equation*}
   of the Khovanov invariant ${\mathcal H}(L)$ is equal to $(Q^{-1}{+}Q)$
   times the Jones polynomial $V_L(Q^2)$ of $L$.

In terms of the associated polynomial $Kh(L)$, 
\begin{equation} \label{JonesfromKhovanov}
  Kh(L)(-1, Q) = (Q^{-1}+Q)V_L(Q^2). \qed
\end{equation} 
 \end{theorem}
 \begin{theorem}[Compare Theorem 1.4 and subsequent remarks from \cite{Lee_Endo04}] \label{Lee4}
   For an alternating knot $L$, its Khovanov invariants ${\mathcal H}^{i,j}(L)$ of 
degree difference $(1,4)$ are paired except in the $0$th cohomology group. 
\qed  \end{theorem}
This fact may be expressed in terms of the polynomial $Kh(L)$, as follows. 
There is another polynomial $Kh'(L)$ in one variable and an  equality 
 \begin{equation} \label{polynomialperiodicity}
   Kh(L)(t, Q) = Q^{-\sigma(L)}\bigl\{ (Q^{-1}{+}Q) + (Q^{-1}+tQ^2\cdot Q)\cdot Kh'(L)(tQ^2) \bigr\}
 \end{equation}
When we combine theorems \ref{Lee1} and \ref{Lee4}, we find that the bi-graded Euler characteristic and
the Jones polynomial of an alternating link determine one another.  Obviously, the equality \eqref{JonesfromKhovanov}
shows that one knows $V_L$ if one knows $Kh(t,Q)$. 

To obtain $Kh(t,Q)$ from $V_L(Q^2)$ requires a certain amount of manipulation.  Implementing these 
manipulations in {\em Maple} and {\em Mathematica} is an important step in our experiments.
 Setting $t{=}-1$ in \eqref{polynomialperiodicity} and combining with equation \eqref{JonesfromKhovanov}, one has
\begin{align*}
  (Q^{-1} + Q) \cdot V_L(Q^2) &= Q^{-\sigma(L)}\bigl\{ (Q^{-1}{+}Q) + (Q^{-1}- Q^3)\cdot Kh'(L)(-Q^2) \bigr\}.
\intertext{Consequently,}
 V_L(Q^2) &= Q^{-\sigma(L)}\bigl\{ 1 + \frac{(Q^{-1}- Q^3)}{(Q^{-1}{+}Q)}\cdot Kh'(L)(-Q^2) \bigr\}
\\
&= Q^{-\sigma(L)}\bigl\{ 1 + (1 - Q^2)\cdot Kh'(L)(-Q^2)\bigr\}.
\intertext{Furthermore,}
  Q^{\sigma(L)} \cdot V_L(Q^2) -1 &= (1 - Q^2)\cdot Kh'(L)(-Q^2),
\intertext{or}
Kh'(L)(-Q^2) &= (1 - Q^2)^{-1}\cdot \bigl(Q^{\sigma(L)} \cdot V_L(Q^2) -1\bigr) .
\end{align*}
Replacing $Q^2$ in the last equation by $-tQ^2$ is the last step to obtain $Kh'(L)$ from the Jones polynomial.
Within a computer algebra system, one must first replace $Q^2$ by $-X$ and then replace $X$ by $tQ^2$.  Once one has
$Kh'(L)(tQ^2)$, one obtains $Kh(t,Q)$ directly from equation~\eqref{polynomialperiodicity}.
\begin{example}
  We have computed $V_{W(3,2)}(t) = t^{-2} - t^{-1} + 1 - t + t^2$ in example \ref{jonesfig8knot}, so
  \begin{align*}
    Kh'\bigl(W(3,2)\bigr)(-Q^2) &= (1- Q^2)^{-1} \cdot \bigl( Q^0 \cdot ( Q^{-4} -  Q^{-2} - Q^2 + Q^4) \bigr)
\\
  &= (1-Q^2)^{-1} \cdot \bigl( (1 - Q^2) \cdot ( Q^{-4} - Q^2) \bigr)
\\
&= Q^{-4} - Q^2.
  \end{align*}
It follows that $Kh'\bigl( W(3,2) \bigr) (tQ^2) = t^{-2}Q^{-4} + tQ^2$, and 
\begin{multline*}
  Kh\bigl(W(3,2)\bigr) (t, Q) = (Q+Q^{-1})+ (Q^{-1} + tQ^3)(t^{-2}Q^{-4} + tQ^2)
\\
 = t^{-2}Q^{-5} + t^{-1}Q^{-1} + Q^{-1} + Q + tQ + t^2Q^5.
\end{multline*}
\end{example}
\section{Khovanov homology examples}  \label{Khovanov}
Once one has the Khovanov polynomial one can make a plot of the Khovanov
homology in an $(i,j)$-plane as in this example. The Betti number $\dim {\mathcal H}^{i,j}\bigl( W(3,11) \bigr)$ is plotted
at the point with coordinates $(i,j)$.
\begin{figure}[h]
  \centering
  \includegraphics[width=6in]{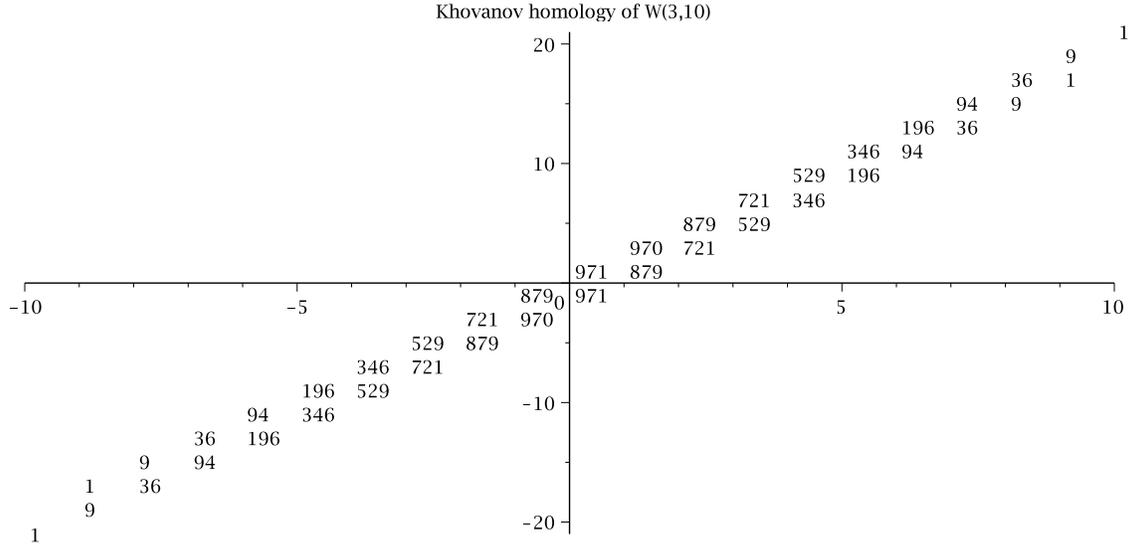}
  \caption{Khovanov homology of $W(3,10)$}
  \label{fig:w310}
\end{figure}
Clearly, as $n$ gets larger, it is going to be harder to make sense of such
plots. Notice that the ``knight move'' $(1,4)$-periodicity
of the Khovanov homology for these knots essentially makes the information on one of the lines $j-2i = \pm 1$ redundant. 

Before we continue to explore Khovanov homology with rational coefficients,
we observe we can also compute  the integral Khovanov homology.   
By Corollary 5 of \cite{Khovanovtorsion} there is only torsion of order 2 in the integral Khovanov homology. 
Even better, there are rules for calculating the number of $\bZ/2\bZ$-summands
present.  We have already noted how  rational Khovanov homology spaces are related by so-called ``knight moves.''  
Except when one lands or starts in a space ${\mathcal H}^{0,*}\bigl(W(3,n)\bigr)$, a move from 
${\mathcal H}^{i,j}\bigl(W(3,n)\bigr)$ to ${\mathcal H}^{i+1,j+4}\bigl(W(3,n)\bigr)$ is a move from one space 
into another space of the same dimension.  
If one reduces the dimensions of the spaces ${\mathcal H}^{0,*}\bigl(W(3,n)\bigr)$ by one, then this phenomenon persists without
qualification on the bidegrees.  Having made this adjustment, 
Shumakovitch \cite[1.G~Definitions]{Khovanovtorsion} provides the following rules for 
computing the  integral Khovanov homology: 
If  $\dim {\mathcal H}^{i,2i+1}\bigl(W(3,n)\bigr) = a \neq 0$, then 
${\mathcal H}^{i,2i+1}\bigl(W(3,n); \bZ \bigr)$ is torsion-free of rank $a$.
Moreover, for every non-zero pair linked by a knight move, 
${\mathcal H}^{i-1,2i-3}\bigl(W(3,n); \bZ \bigr)$ and ${\mathcal  H}^{i,2i+1}\bigl(W(3,n); \bZ \bigr)$ have the same 
rank, but groups along the line $j{-}2i = -1$ may have torsion.  
In fact,  the two-torsion part of ${\mathcal H}^{i,2i-1}\bigl(W(3,n)\;\bZ \bigr)$ 
is an abelian $2$-group $(\bZ/2\bZ)^a$.  The table \ref{IKHW34} shows the rules in operation
for the knot $W(3,4)$, also identified as 8\_18 in standard knot tables. Integer entries along 
the line $j{-}2i = 1$ indicate
the ranks of free abelian groups and an entry $r, a_2$ along the line
$j {-} 2i= -1$ indicates a free abelian group of rank $r$ summed with a 2-group $(\bZ/2\bZ)^a$.
\begin{table}
\caption{Integral Khovanov homology of $W(3,4)$}\label{IKHW34}
\begin{center}
  \begin{tabular}[h]{||c|c|c|c|c|c|c|c|c|c||}  \hline \hline 
    & -4 & -3 & -2 & -1 & 0 & 1 &2 & 3 & 4
\\ \hline
9  &  &  &    &    &    &   &   &  &   1
\\ \hline
7  &  &  &    &    &    &   &   & 3  &  $1_2$
\\  \hline
5  &  &  &  &      &    &    & 3  & 1, $3_2$  &    
\\ \hline
3 &  &  &   &      &    &  4  & 3, $3_2$  &   &    
\\ \hline
1 &  &  & & & 5  & 3 ,$4_2$  &   &   &   
\\  \hline
-1 &  &  &  & 3 & 5, $4_2$   &    &   &   &    
\\ \hline
-3 & &  & 3  & 4, $3_2$ &    &    &   &   &    
\\  \hline
-5 & & 1 &  3, $3_2$ & &    &    &   &   &  
\\  \hline
-7 &  & 3, $1_2$ &  &    &    &    &   &   &  
\\  \hline
-9 & 1  &  &  &    &    &    &   &   &  
\\  \hline \hline
  \end{tabular}
\end{center}
\end{table}

Returning to rational Khovanov homology,
we take advantage of the ``knight move'' periodicity and simplify by recording the Betti numbers from only along the line $j- 2i = 1$. 
In order to study the asymptotic behavior of Khovanov homology we have to normalize the data.  This is done
by computing the total rank of the Khovanov homology along the line and dividing each Betti number by the
total rank.  We obtain normalized Betti numbers that sum to one.  

This raises the possibility of approximating the distribution of normalized Betti numbers by a probability distribution.
For our baseline experiments we choose to use the normal $N(\mu, \sigma^2)$ probability density function
\begin{equation*}
  f_{\mu, \sigma^2}(x) = \frac{1}{\sigma \sqrt{2\pi}} \exp \Bigl( - \frac{(x-\mu)^2}{2\sigma^2} \Bigr) 
\end{equation*}
Fit a quadratic function 
$q_n(x)= -(\alpha \, x^2 - \beta\, x + \delta)$ 
to the logarithms of the normalized Khovanov dimensions along the line $j=2i+1$
and exponentiate the quadratic function. 
Since the total of the normalized dimensions is 1, we normalize the
exponential, obtaining 
\begin{equation*}
  \rho_n(x) = A_n e^{q_n(x)} \quad \text{satisfying} \quad \int_{-\infty}^{\infty} \rho_n(x) \; dx = 1.
\end{equation*}
To obtain a formula for $A_n$, complete the square
\begin{equation*}
  q_n(x) = -\alpha \cdot \bigl( x - (\beta/2\alpha) \bigr)^2 +\bigl( (\beta^2/4\alpha) - \delta\bigr).
\end{equation*}
Then consider
\begin{align*}
  1 &= A_n \int_{-\infty}^{\infty} \exp q_n(x) \; dx
\\
    &= A_n \cdot 
 \int_{-\infty}^{\infty} \exp \bigl((\beta^2/4\alpha) - \delta \bigr) 
        \cdot \exp \bigl( -\alpha \cdot \bigl( x -(\beta/2\alpha) \bigr)^2\bigr)  \; dx
\\
   &= A_n \cdot \exp \bigl((\beta^2/(4\alpha) - \delta \bigr)  \cdot
           \int_{-\infty}^{\infty} \exp \bigl( -\alpha \cdot \bigl( x -(\beta/2\alpha) \bigr)^2\bigr)  \; dx
\\
   &= A_n \cdot  \exp \bigl((\beta^2/4\alpha)  - \delta\bigr)  \cdot  \sqrt{\pi/\alpha}
\end{align*}
% \begin{equation*}
% \exp( q_n(x) ) =   \exp \bigl( \frac{\beta^2}{4\alpha} - \delta\bigr) \cdot \exp -\alpha \cdot \bigl( x - \frac{\beta}{2\alpha} \bigr)^2
% \end{equation*}
%Since $ \int_{-\infty}^{\infty} \exp(  \alpha \cdot - \bigl( x - \frac{\beta}{2\alpha} \bigr)^2 \; dx = \sqrt{\frac{\pi}{\alpha}}$
Thus, the expression for $A_n$ is 
\begin{equation*}
  A_n =  \exp\bigl( -\bigl((\beta^2/4\alpha) - \delta\bigr)\bigr) \cdot \sqrt{\alpha/\pi}.
\end{equation*}

Equate the expressions
\begin{equation*}
  \rho_n(x) = \frac{1}{\sigma_n \sqrt{2\pi}} \exp \Bigl( - \frac{(x-\mu_n)^2}{2\sigma_n^2} \Bigr) 
\quad \text{and} \quad
\rho_n(x) = A_n \exp ( q_n(x)),
\end{equation*}
and observe $\mu_n = \beta/2\alpha$ by equating the two expressions for the location of the local maximum
of $\rho_n(x)$. Then the efficient way to the parameter $\sigma_n$ is to solve the equation
\begin{equation*}
\frac{1}{\sigma_n \sqrt{2\pi}} = \rho_n(\beta/2\alpha) = A_n \exp( q_n(\beta/2\alpha)) 
    =  \exp\bigl( -\bigl((\beta^2/4\alpha) - \delta\bigr)\bigr) \cdot \sqrt{\alpha/\pi} \cdot
\exp\bigl( (\beta^2/4\alpha)  - \delta \bigr),
\end{equation*}
obtaining $\sigma_n = 1/ \sqrt{2\alpha} $.

Working this out for $W(3,10)$,  and carrying only 3 decimal places, 
the raw dimensions are
\begin{center}
  \begin{small}
     \begin{tabular}[h!]{c|c|c|c|c|c|c|c|c|c|c}
$i$  & -9 & -8 & -7 & -6 &  -5 &  -4 & -3&    -2 &-1 &  0
\\
$\dim$  &    1& 9& 36& 94& 196& 346& 529& 721& 879& 970
\\
 $i$  & 1& 2& 3& 4& 5& 6& 7& 8& 9& 10  
\\
$\dim$  & 971& 879& 721& 529& 346& 196& 94& 36& 9& 1 
  \end{tabular}
  \end{small}
\end{center}
and, to three significant digits, the logarithms of the normalized dimensions are
\begin{center}
  \begin{small}
    \begin{tabular}[h!]{c|c|c|c|c|c|c|c|c|c|c}
$i$  & -9 & -8 & -7 & -6 &  -5 &  -4 & -3&    -2 &-1 &  0
\\
    &  -17.9& -15.7& -14.3& -13.3& -12.6& -12.0& -11.6& -11.3& -11.1& -11.0
\\
 $i$  & 1& 2& 3& 4& 5& 6& 7& 8& 9& 10  
\\
& -11.0& -11.1& -11.3& -11.6& -12.0& -12.6& -13.3& -14.3& -15.7& -17.9                 
\end{tabular}
  \end{small}
\end{center}
Fitting a quadratic to this information,  we get
\begin{equation*}
          q_{10}(x) =  -10.7 + 0.0720\, x - 0.0720\, x^2,\quad \alpha = \beta =   0.0720, \quad \delta =     10.7.
\end{equation*}
To three significant digits $\mu_{10} = 0.500$ and $\sigma_{10} = 2.64$.

By the symmetry of Khovanov homology, the mean $\mu_n$ approaches $1/2$ rapidly, so this parameter is of little interest.
On the other hand, relating the parameter $\sigma_n$  to some geometric quantity, say, some hyperbolic invariant 
of the complement of the link, is a very interesting problem. 

For $W(3,10)$, the density function is
\begin{multline*}
  \rho_{10}(x) = 11686.8431618280538\,\sqrt {{\pi }^{-1}}  \\
\cdot \exp ({- 10.7018780565714309+ 0.0716848579220777243\,x- 0.0716848579220778631\,{x}^{2}})
\end{multline*}
When placed into standard form, $\mu_{10} =0.5000054030 $ and $\sigma_{10} = 2.640882970 $.
Figure \ref{fig:w310comp} compares the plot of the density function $\rho_{10}$ with a point plot of normalized dimensions.
\begin{figure}[h]
  \centering
  \includegraphics[height=2.0in]{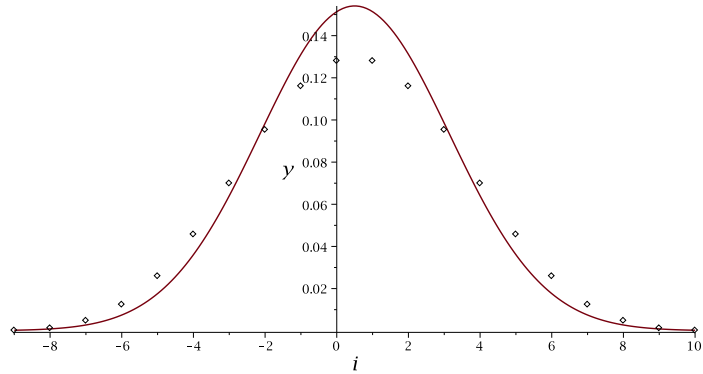}
  \caption{normalized homology of $W(3,10)$ compared with density function}
  \label{fig:w310comp}
\end{figure}
\newline
For the knot $W(3,11)$ 
 the expression for the density function is
\begin{multline*}
  \rho_{11}(x) =29676.8676257830375\,\sqrt {{\pi }^{-1}} \\
\cdot \exp ({- 11.6724860231789886+ 0.0661625395821569817\,x- 0.0661623073574252735\,{x}^{2}})
\end{multline*}
%\enlargethispage{10pt}
When placed into standard form, $\mu_{11} =0.5000017550 $ and $\sigma_{11} = 2.749031276$.
Figure \ref{fig:w311comp} compares the plot of the density function $\rho_{11}$ with a point plot of normalized dimensions.
\begin{figure}[h!]
  \centering
  \includegraphics[height=2.0in]{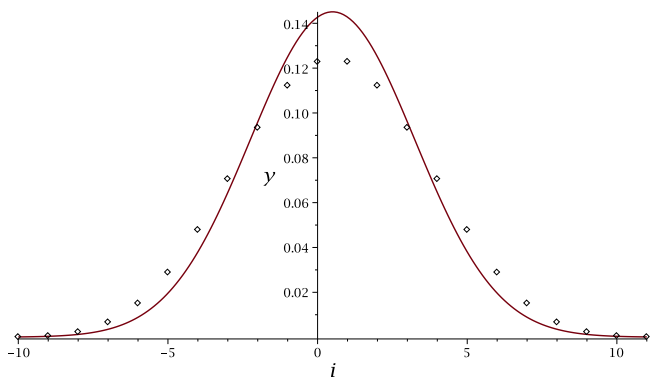}
  \caption{normalized homology of $W(3,11)$ compared with density function}
  \label{fig:w311comp}
\end{figure}
\newline
For $W(3,22)$, the density function is
\begin{multline*}
  \rho_{22}(x) =833596689.149608016\,\sqrt {{\pi }^{-1}} \\
\cdot \exp ({- 22.2219365040983057+ 0.0353061029354434300\,x- 0.0353061029347388616\,{x}^{2}})
\end{multline*}
When placed into standard form, $\mu_{22} =0.500000000 $ and $\sigma_{22} = 3.763224354$.
Figure \ref{fig:w322comp} compares the plot of the density function $\rho_{22}$ with a point plot of normalized dimensions.
\begin{figure}[h!]
  \centering
  \includegraphics[height=2.0in]{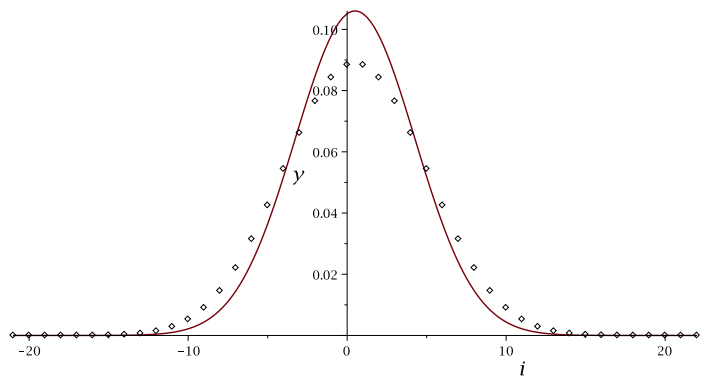}
  \caption{normalized homology of $W(3,22)$ compared with density function}
  \label{fig:w322comp}
\end{figure}
\newline
For $W(3,23)$, the density function is
\begin{multline*}
  \rho_{23}(x) = 2113964949.23002362\,\sqrt {{\pi }^{-1}} \\ 
\cdot exp( {- 23.1731352596503442+ 0.0338545815354610105\,x- 0.0338545815348441914\,{x}^{2}})
\end{multline*}
When placed into standard form, $\mu_{23} =0.5000000000 $ and $\sigma_{23} =3.843052143 $.
Figure \ref{fig:w323comp} compares the plot of the density function $\rho_{23}$ with a point plot of normalized dimensions.
\begin{figure}[h!]
  \centering
  \includegraphics[height=2.0in]{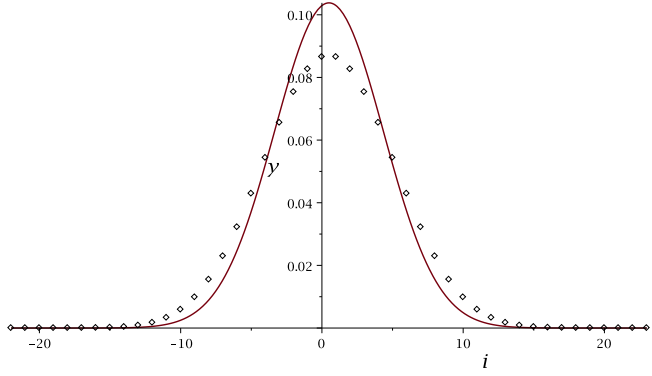}
  \caption{normalized homology of $W(3,23)$ compared with density function}
  \label{fig:w323comp}
\end{figure}
\newline
{\em Maple} worksheets and, later, {\em Mathematica} notebooks will be available at URL \cite{software} prepared by the second-named author.
\section{Data Tables}\label{Data}
This section contains tables of data generated using {\em Maple} to implement some of the results of earlier sections.
The first table collects data for weaving knots $W(3,n)$ with $n \equiv 1 \mod 3$; the second table does the same for
weaving knots $W(3,n)$ with $n \equiv 2 \mod 3$. 
In each table the first column lists the value of $n$; 
the second column lists the total dimension of the Khovanov homology lying along the line $j = 2i{+}1$; and the third
column lists the dimension of the vector space ${\mathcal H}^{0,1}\bigl( W(3,n)\bigr)$. 
Columns four and five display measures of the deviation of the proposed normal distributions from 
the actual distribution of normalized dimensions.

In section \ref{Khovanov} we have approximated a distribution of normalized
Khovanov dimenstions by a standard normal distribution, and we have displayed
graphics comparing an actual distribution with its approximation. 
To quantify those visual impressions, we compute and tabulate an $L^1$- and an $L^2$-deviation.  Let 
\begin{equation*}
  \text{Total dimension} = \sum_{i=-2n}^{2n+1} \dim {\mathcal H}^{i, 2i+1}\bigl( W(3,n) \bigr).
\end{equation*}
For the $L^2$-comparison, we compute
\begin{equation*}
 \Biggl(  \sum_{i = -2n}^{2n+1} \biggl( \rho_n(i)- \frac{\dim {\mathcal H}^{i, 2i+1}\bigl( W(3,n) \bigr)}{\text{Total dimension}} \biggr)^2 \Biggr)^{1/2}
\end{equation*}
For the $L^1$-comparison, we compute 
\begin{equation*}
   \sum_{i = -2n}^{2n+1} \Babs{ \rho_n(i) - \frac{\dim {\mathcal H}^{i, 2i+1}\bigl( W(3,n) \bigr)}{\text{Total dimension} } }
\end{equation*}
The $L^2$ comparisons appear to tend to 0, whereas the $L^1$ comparisons
appear to be growing slowly.  
\begin{table} \caption{Data for $W(3,n)$ with $n \equiv 1 \mod 3$}
  \begin{tabular}[h!]{|c|c|c|c|c|c|}
    $n$ & Total dimension & $\dim {\mathcal H}^{0,1}$  & $\sigma$  &  $L^2$-comparison & $L^1$ comparison
\\
    10 & 7563 &   970 & 2.64088  &   0.040510 &  0.134828
\\
    13 & 135721 & 15418 &  2.95616 & 0.041133 & 0.150599
\\
    16 & 2435423 &  250828 & 3.24564 &  0.040792 & 0.155995
\\
    19  & 43701901 & 4146351 & 3.51339 & 0.040145 & 0.161336
\\
    22  & 784198803 & 69337015 &  3.76322 &  0.039413 & 0.165763
\\
    25  & 14071876561 &  1169613435 & 3.99810 & 0.038678 & 0.167576
\\
    28  &  252509579303 & 19864129051 & 4.22032  & 0.037971 & 0.167790
\\
    31  &  4531100550901 & 339205938364 & 4.43167 & 0.037303 & 0.170736
\\
    34  &  81307300336923 & 5818326037345 & 4.63358 & 0.036676 & 0.172392
\\
    37  & 1459000305513721 &  100173472277125 & 4.82719 & 0.036089 & 0.173119
\\
    40  & 26180698198910063  & 1730135731194046  & 5.01342 & 0.035541 & 0.173178
\\ 
    43  & 469793567274867421 &  29963026081609060  &  5.19306 &  0.035028 & 0.173812
\\
    46  & 8430103512748703523 & 520131503664409798 & 5.36674 & 0.034546 & 0.175052
\\
    49  & $1.51272\cdot 10^{20}$ & $ 9.04765\cdot10^{18}$  &  5.53502 & 0.034093 & 0.175779
\\
    52  &  $ 2.71447\cdot 10^{21}$  &  $ 1.57670\cdot 10^{20} $ & 5.69838 & 0.033667 & 0.176100
\\
    55  &  $ 4.87091\cdot 10^{22} $ & $ 2.75210\cdot 10^{21} $ & 5.85721 & 0.033265 & 0.176098
\\
    58  &  $ 8.74050\cdot 10^{23} $ & $ 4.81071\cdot 10^{22} $   & 6.01187 & 0.032885 & 0.175898
\\
    61  &  $ 1.56842\cdot 10^{25} $& $  8.42017\cdot 10^{23}$ & 6.16267 & 0.032524 & 0.176778
\\
    64  &  $ 2.81441\cdot 10^{26} $ & $ 1.47552\cdot 10^{25} $ & 6.30989 & 0.032182 & 0.177369
\\
    67  &  $5.05026\cdot 10^{27}$  & $ 2.58843\cdot 10^{26} $ & 6.45376 & 0.031857 & 0.177716
\\
    70  &$ 9.06233\cdot 10^{28} $& $4.54520\cdot 10^{27}$ & 6.59451 & 0.031547 & 0.177859
\\
    73  & $1.62617\cdot 10^{30}$ & $7.98842\cdot 10^{28}$ & 6.73233 & 0.031251 & 0.177831
\\
    76  & $ 2.91804\cdot 10^{31}$ & $1.40517 \cdot 10^{30}$ & 6.86740 & 0.030968 & 0.177657
\\
    79  &  $5.23621\cdot 10^{32}$ & $ 2.47359 \cdot 10^{31}$ & 6.99986 & 0.030697 & 0.177995
\\
    82  & $9.39600\cdot 10^{33} $&  $4.35747 \cdot 10^{32}$& 7.12988 & 0.030437 & 0.178445
\\
    85  &  $1.68604\cdot 10^{35} $&  $7.68116 \cdot 10^{33}$ & 7.25757 & 0.030188 & 0.178746
\\
    88  & $ 3.02548\cdot 10^{36} $ &  $1.35483 \cdot 10^{35}$ & 7.38305 & 0.029948 & 0.178918
\\
    91  &  $ 5.42901\cdot 10^{37} $ &  $  2.39106 \cdot 10^{36} $ & 7.50645 & 0.029718 & 0.178976
\\
   94  & $ 9.74196\cdot 10^{38}$ & $ 4.22211 \cdot 10^{37} $ & 7.62786 & 0.029496 & 0.178935
\\
    97 &  $ 1.74812\cdot 10^{40} $ & $7.45910  \cdot 10^{38}$ & 7.74736 & 0.029282 & 0.178807
\\
100  & $3.13688\cdot 10^{41} $ &  $1.31840   \cdot 10^{40} $& 7.86506 & 0.029075 & 0.178890
\\
121  &  $1.87923\cdot 10^{50} $&   $ 7.18477\cdot 10^{48} $ & 8.64424 & 0.027805 & 0.179577
\\
142  &  $1.12580\cdot 10^{59} $ & $ 3.97500 \cdot 10^{57} $ & 9.35886& 0.026769& 0.180247
\\
163  & $6.74436\cdot 10^{67} $   &  $2.22337 \cdot 10^{66}$ & 10.0227& 0.025900& 0.180596
\\
184  & $4.04037\cdot 10^{76}$ & $ 1.25398 \cdot 10^{75} $ & 10.6453& 0.025156& 0.180629
\\
205  & $ 2.42049\cdot 10^{85}$ &  $7.11854 \cdot 10^{83} $ & 11.2334& 0.024508& 0.180907
\\
247 & $8.68689\cdot 10^{102}$ & $2.32816 \cdot 10^{101} $ & 12.3258& 0.023423& 0.181027
\\
289 & $3.11764\cdot 10^{120} $ & $7.72623 \cdot 10^{118} $ & 13.3289& 0.022542& 0.181268
  \end{tabular}
\end{table}
\begin{table} \caption{Data for $W(3,n)$ with $n \equiv 2 \mod 3$}
\begin{tabular}{c|c|c|c|c|c}
    $n$ & Total dimension & $\dim {\mathcal H}^{0,1}$  & $\sigma$  &    $L^2$-comparison & $L^1$ comparison
\\
    11 & 19801&2431& 2.74903& 0.040906&  0.141925
\\
  14&355323&38983& 3.05533& 0.041079&  0.153170
\\
 17&6376021&637993& 3.33710& 0.040595& 0.156595
\\
20&114413063&10591254& 3.59850& 0.039905& 0.163190
\\
23&2053059121&177671734& 3.84305& 0.039166& 0.166596
\\
26&36840651123&3004390818& 4.07348& 0.038438& 0.167789
\\
29&661078661101&51124396786& 4.29190& 0.037744& 0.168941
\\
 32&11862575248703&874400336044& 4.49997 & 0.037089& 0.171411
\\
35&212865275815561&15018149469823& 4.69899& 0.036476& 0.172723
\\
38&3819712389431403&258853011125599& 4.89004& 0.035903& 0.173203
\\
41&68541957733949701&4474997964407374& 5.07400& 0.035366& 0.173083
\\
44&1229935526821663223& 77563025486587315& 5.25158& 0.034864& 0.174290
\\
47&22070297525055988321& 1347390412214087833 & 5.42341 & 0.034392& 0.175346
\\
50& $ 3.96035\cdot 10^{20} $& $ 2.34525 \cdot 10^{19} $ & 5.59000& 0.033949& 0.175926
\\
53& $ 7.10657\cdot 10^{21} $& $ 4.08927 \cdot 10^{20} $& 5.75181 & 0.033531 & 0.176131
\\
56& $ 1.27522\cdot 10^{23} $& $ 7.14133\cdot 10^{21} $& 5.90921 & 0.033136 & 0.176037
\\
59& $ 2.28829\cdot 10^{24} $& $ 1.24888 \cdot 10^{23} $& 6.06255& 0.032763 & 0.176227
\\
62& $ 4.10617\cdot 10^{25} $& $ 2.18679\cdot 10^{24} $& 6.21213& 0.032408& 0.177005
\\
65& $ 7.36823\cdot 10^{26} $& $ 3.83347\cdot 10^{25} $& 6.35821& 0.032072& 0.177510
\\
68& $ 1.32218\cdot 10^{28} $& $ 6.72713\cdot 10^{26} $& 6.50102& 0.031752& 0.177785
\\
71& $ 2.37255\cdot 10^{29} $& $ 1.18163\cdot 10^{28} $& 6.64077 & 0.031446 & 0.177867
\\
74& $ 4.25736\cdot 10^{30} $& $ 2.07736\cdot 10^{29} $& 6.77765 & 0.031155 & 0.177787
\\
77& $ 7.63953\cdot 10^{31} $& $ 3.65504 \cdot 10^{30} $& 6.91183 & 0.030876 & 0.177602
\\
80& $ 1.37086\cdot 10^{33} $& $ 6.43571 \cdot 10^{31} $& 7.04347 & 0.030609 & 0.178163
\\
83& $ 2.45990\cdot 10^{34} $& $ 1.13397 \cdot 10^{33} $& 7.17269 & 0.030353 & 0.178561
\\
86& $ 4.41412\cdot 10^{35} $& $ 1.99933 \cdot 10^{34} $& 7.29963 & 0.030107 & 0.178817
\\
89& $ 7.92082\cdot 10^{36} $& $ 3.52717\cdot 10^{35} $& 7.42441 & 0.029871 & 0.178949
\\
92& $ 1.42133\cdot 10^{38} $& $ 6.22605\cdot 10^{36} $& 7.54714 & 0.029643& 0.178972
\\
95& $ 2.55048\cdot 10^{39} $& $ 1.09958\cdot 10^{38} $& 7.66790 & 0.029424 & 0.178901
\\
98& $ 4.57665\cdot 10^{40} $& $ 1.94290\cdot 10^{39} $& 7.78679 & 0.029212 & 0.178747
\\
 119& $ 2.74175\cdot 10^{49} $& $ 1.05696\cdot 10^{48} $& 8.57308& 0.027914& 0.179650
\\
140& $ 1.64251\cdot 10^{58} $& $ 5.84051 \cdot 10^{56} $& 9.29316& 0.026859& 0.180257
\\
161& $ 9.83989\cdot 10^{66} $& $ 3.26385\cdot 10^{65} $& 9.96138& 0.025977& 0.180552
\\
182& $ 5.89483\cdot 10^{75} $& $ 1.83951\cdot 10^{74} $& 10.5875& 0.025223& 0.180539
\\
203& $ 3.53144\cdot 10^{84} $& $ 1.04367\cdot 10^{83} $& 11.1787& 0.024566& 0.180926
\\
245& $ 1.26740\cdot 10^{102} $& $ 3.41053\cdot 10^{100} $& 12.2759& 0.023469& 0.181064
\\
287& $ 4.54858\cdot 10^{119} $& $ 1.13115\cdot 10^{118} $& 13.2829& 0.022580& 0.181221
\\
329& $ 1.63244\cdot 10^{137} $& $ 3.79224\cdot 10^{135} $& 14.2187& 0.021838& 0.181399
\end{tabular}  
\end{table}
\newpage
\section{More on  Polynomials} \label{PolynomialsExtra}
First we  collect basic values of polynomials $C_{n,*}(q)$, some of which are
referred to in sections  \ref{Hecke}, \ref{Polys}, and \ref{TwistNumbersVolume}.  
Recall the initilization values:
\begin{equation*} 
  C_{1,0}(q) = 0, \; C_{1,1}(q) = -(q{-}1), \; C_{1,2}(q) = 0, \; C_{1,12}(q) = 1,
                   \;  C_{1,21}(q) = 0, \; \text{and} \; C_{1,121}(q) = 0.
\end{equation*}
\begin{align*}
    C_{2,0}(q) &=  q(q{-}1)^2 = q - 2\,q^2 + q^3   &   C_{3,0}(q) &= -q + 4\, q^2 - 5\,q^3 + 4\,q^4 - q^5   
\\
  C_{2,1}(q) &= (q{-}1)^3 =  - 1 + 3\, q  - 3\,q^2 +  q^3   &   C_{3,1}(q) &= 1 - 4\, q + 7\,q^2 - 7q^3 + 4\, q^4 - q^5 
\\
  C_{2,2}(q) &= -q(q{-}1) = q - q^2     &   C_{3,2}(q) &= - q + 3\, q^2 - 3\, q^3 + q^4  
\\
  C_{2,12}(q) &= -(q{-}1)^2 =  -1 + 2\,q -q^2   &   C_{3,12}(q) &= 1 - 3\, q + 4\,q^2 - 3\,q^3 + q^4  
\\
  C_{2,21}(q) &= q    &    C_{3,21}(q) &= -q + 2\,q^2 - q^3  
\end{align*}
\begin{align*}
  C_{4,0}(q) &= q- 5\,q^2+ 10 \,q^3 -12\,q^4 +10\,q^5 - 5\,q^6 + q^7
\\
  C_{4,1}(q) &= -1 + 5\,q - 11 \, q^2 + 16\, q^3 - 16\, q^4 + 11\, q^5 - 5 \, q^6 + q^7
\\
  C_{4,2}(q) &= q - 4\,q^2 + 7\,q^3 - 7\,q^4 + 4\,q^5 - q^6
\\
  C_{4,12}(q) &= -1 + 4\,q - 7\,q^2 + 9\,q^3 - 7\,q^4 + 4 \,q^5 - q^6
\\
  C_{4,21}(q) &= q - 3\,q^2 + 4\,q^3 - 3 \, q^4 + q^5
\end{align*}
  \begin{table}[h!]
   \caption{Alexander polynomials for $W(3,n)$}
    \label{alexpolys}
    \centering  \renewcommand{\arraystretch}{1.25}
    \begin{tabular}{|c|c|} \hline
      $n$ & $\Delta_{W(3, n)}(t)$
\\ \hline
4 &  $-t^{3}+5\,t^{2}-10\,t+13-10\,t^{-1}+5\,t^{-2}-t^{-3}$
\\ \hline
5 &  $t^{4}-6\,t^{3}+15\,t^{2}-24\,t+29-24\,t^{-1}+15\,t^{-2}-6\,t^{-3}+t^{-4}$
\\ \hline
10 &  $ -t^{9}+11\,t^{8}-55\,t^{7}+174\,t^{6}-409\,t^{5}+777\,t^{4} -1243\,t^{3}$
\\ 
   & $+1716\,t^{2}-2073\,t+2207-2073\,t^{-1}+1716\,t^{-2}$
\\
&  $-1243\,t^{-3}+777\,t^{-4}-409\,t^{-5}+174\,t^{-6}-55\,t^{-7}+11\,t^{-8}-t^{-9}$
\\ \hline
11 &  $t^{10}-12\,t^{9}+66\,t^{8}-230\,t^{7}+593\,t^{6}-1232\,t^{5}+2157\,t^{4}-3268\,t^{3}$
\\
   & $+4356\,t^{2}-5158\,t+5455 -5158\,t^{-1}+4356\,t^{-2}$
\\
   & $ -3268\,t^{-3}+2157\,t^{-4}-1232\,t^{-5} +593\,t^{-6}-230\,t^{-7}+66\,t^{-8}-12\,t^{-9}+t^{-10}$
\\ \hline
    \end{tabular}
   \end{table}
\begin{table}[h!]
  \caption{Jones polynomials for $W(3,n)$}
  \label{Jonesw3n}
  \centering   \renewcommand{\arraystretch}{1.25}
  \begin{tabular}{|c|c|} \hline
$n$  &   $V_{W(3,n)}(t)$ 
\\  \hline
4 &  $t^{4}-4\,t^{3}+6\,t^{2}-7\,t+9-7\,t^{-1}+6\,t^{-2}-4\,t^{-3}+t^{-4}$
\\   \hline
5 &  $-t^{5}+5\,t^{4}-10\,t^{3}+15\,t^{2}-19\,t+21-19\,t^{-1}+15\,t^{-2}-10\,t^{-3}+5\,t^{-4}-t^{-5}$
\\  \hline
10 &  $t^{10}-10\,t^{9}+45\,t^{8}-130\,t^{7}+290\,t^{6}-542\,t^{5}+875\,t^{4}$
\\  
&   $-1250\,t^{3}+1600\,t^{2}-1849\,t+1941-1849\,t^{-1} +1600\,t^{-2}-1250\,t^{-3}$ 
\\  
 &   $+875\,t^{-4}-542\,t^{-5}+290\,t^{-6}-130\,t^{-7}+45\,t^{-8}-10\,t^{-9}+t^{-10}$
\\ \hline
11 &  $-t^{11}+11\,t^{10}-55\,t^{9}+176\,t^{8}-429\,t^{7}+869\,t^{6}-1518\,t^{5}+2343\,t^{4}$
\\
 & $-3245\,t^{3}+4070\,t^{2}-4652\,t+ 4863-4652\,t^{-1}+4070\,t^{-2}-3245\,t^{-3}$
\\
 & $+2343\,t^{-4}-1518\,t^{-5}+869\,t^{-6}-429\,t^{-7}+176\,t^{-8}-55\,t^{-9}+11\,t^{-10}-t^{-11}$
\\ \hline
  \end{tabular}
\end{table}
\newpage
\begin{table}[h!]
  \caption{HOMFLY-PT polynomials for $W(3,n)$}
  \label{homflyptw3n}
  \centering  \renewcommand{\arraystretch}{1.25}
  \begin{tabular}{|c|c|}  \hline
    $n$  &   $H_{W(3,n)}(a, z)$ 
\\  \hline
4  &  $a^2\bigl(  {z}^{4}+{z}^{2}-1 \bigr)
   + \bigl(-{z}^{6}-3\,{z}^{4}-{z}^{2}+3   \bigr)
+ a^{-2}\bigl( {z}^{4}+{z}^{2}-1 \bigr)$
\\  \hline
5 & $ a^2\bigl( -{z}^{6}-2\,{z}^{4}+{z}^{2}+2  \bigr)   
 +  \bigl( {z}^{8}+4\,{z}^{6}+3\,{z}^{4}-4\,{z}^{2}-3\bigr)  
  + a^{-2}\bigl( -{z}^{6}-2\,{z}^{4}+{z}^{2}+2  \bigr)$
\\ \hline
10 & $a^2\bigl( {z}^{16}+7\,{z}^{14}+14\,{z}^{12}-2\,{z}^{10}-29\,{z}^{8}-11\,{z}^{6}+ 18\,{z}^{4}+6\,{z}^{2}-3  \bigr)$  
\\
   & $+  \bigl( -{z}^{18}-9\,{z}^{16}-28\,{z}^{14}-26\,{z}^{12}+33\,{z}^{10}+69\,{z}^{8}+4\,{z}^{6}-42\,{z}^{4}-9\,{z}^{2}+7\bigr)$
\\
   &$  + a^{-2}\bigl( {z}^{16}+7\,{z}^{14}+14\,{z}^{12}-2\,{z}^{10}-29\,{z}^{8}-11\,{z}^{6}+ 18\,{z}^{4}+6\,{z}^{2}-3  \bigr)$
\\ \hline
11 &  $a^2\bigl( -{z}^{18}-8\,{z}^{16}-20\,{z}^{14}-6\,{z}^{12}+40\,{z}^{10}+34\,{z}^{8}-25\,{z}^{6}-24\,{z}^{4}+6\,{z}^{2}+4  \bigr)   $
\\
 &  $+  \bigl( {z}^{20}+10\,{z}^{18}+36\,{z}^{16}+46\,{z}^{14}-28\,{z}^{12}-114\,{z}^{10}-43\,{z}^{8}+74\,{z}^{6}+42\,{z}^{4}-16\,{z}^{2}-7 \bigr) $
\\
 & $ + a^{-2}\bigl( -{z}^{18}-8\,{z}^{16}-20\,{z}^{14}-6\,{z}^{12}+40\,{z}^{10}+34\,{z}^{8}-25\,{z}^{6}-24\,{z}^{4}+6\,{z}^{2}+4  \bigr)$
\\ \hline
  \end{tabular}
\end{table}
\section{Notes on Computing} 
\label{ComputerNotes}
This file contains some remarks on the roles played by {\em Mathematica} and
{\em Maple} experiments in generating data, conjectures, and results.

Our initial interest was in the Khovanov homology of weaving knots, which we knew was determined in a straightforward 
manner by the Jones polynomials.  It also turns out that the Khovanov homology of our knots can be determined by knowing
half of the Khovanov homology, essentially.  Instead of having to keep track of a bigraded object, the study of Khovanov
homology of weaving knots is reduced to the study of a graded object.  We normalized our examples by dividing each dimension
in the graded object by the total dimension and plotted the results for a large number of the knots.  In the plots
bell-shaped curves appear as envelopes of the plots of the normalized dimensions.  First, this led us to conjecture 
that the standard deviations of the bell curves may be an interesting invariant for the family of knots $W(3, n)$. 

As mentioned, the Jones polynomial of a weaving knot $W(3,n)$ determines the two-variable Khovanov polynomial of
the bi-graded Khovanov homology.  To simplify matters, we studied the Jones polynomial on its own terms.  
We knew that the Jones polynomials have the form
\begin{equation*}
  V_{W(3,n)}(t) = \pm t^{-n} + \lambda_{-n+1}t^{-n+1} + \cdots \lambda_{n-1}t^{n-1} + \pm t^n,
\end{equation*}
so we conjectured that $\lambda_{-n+k} = \lambda_{n-k}$ is a polynomial
function of degree $k$ in $n$. The basis for this conjecture is the well-known binomial distribution approximating
the standard normal distribution.  To investigate
this conjecture further, we detoured through another round of experiments. 

During a visit to the University of Osnabr\"{u}ck in Germany, the second-named author was tutored in {\em Mathematica} 
by Prof.~Dr.~Karl-Heinz Spindler.  During the demonstrations of techniques for manipulating polynomials, Dr.~Spindler 
asked if we knew explanations of the patterns we were observing.  These questions led to the formulation and 
proof of the palindromic properties of the building  block polynomials $C_{n,-}(q)$ stated in theorem \ref{palindromes}.

For a large sample of computed Jones polynomials, we extracted the coefficients $\lambda_{-n+k}$ obtaining sequences of 
integers upon which {\em Mathematica} routines computed iterated differences.  In accordance with the conjectured 
behavior, we observed the differences vanishing after the expected number of iterations.  From these experiments
it was  possible to generate formulas for the numbers $\lambda_{-n+2}$ and $\lambda_{-n+3}$
eventually proved in theorems \ref{twistnumber2} and \ref{twistnumber3}.

One may also obtain expressions for 
two-variable HOMFLY-PT polynomial $H_{(W(3,n)}(a, z)$ normalized to $H({\rm Unknot})(a,z) = 1$.
This amounts to applying a different sequence of substitutions to the Hecke algebra output
$ V_{(\sigma_1\sigma_2^{-1})^n}(q,z) $ given in \eqref{Heckeoutput}.
Since we have no immediate use for these gadgets, we offer the brief table \ref{homflyptw3n}. 

Turning to the volume computations, we thank Ilya Kofman for a significant
improvement of our first script for computing volumes using SnapPy \cite{SnapPy}.  According to our 
data, the volume of the complement of $W(3,n)$ is growing roughly linearly with $n$, so it is not so 
surprising that the volume is strongly correlated with the higher twist numbers.   Another feature 
of the SnapPy data is that, although the volume is growing linearly, the number of simplices used 
by SnapPy to compute the volume is growing quite irregularly.  
\bibliographystyle{plain}
\bibliography{knotinvariantslist}
\end{document}